\documentclass[12pt,a4paper]{amsart}

\usepackage{tikz-cd} % for diagrams
\tikzcdset{row sep/normal=3.6em, column
	sep/normal=3.6em}
\usetikzlibrary{decorations.pathmorphing}
\usepackage{quiver}

\usepackage{amssymb}
\usepackage{mathrsfs} % for mathscr (names of categories) (only capital letters)
\usepackage{enumitem} % for enumerate styles
\usepackage{moreenum}
\setlist{wide}
\usepackage{caption}

\usepackage{stmaryrd}

\newtheorem{theorem}{Theorem}[section]
\newtheorem{proposition}[theorem]{Proposition}
\newtheorem{lemma}[theorem]{Lemma}

\theoremstyle{definition}
\newtheorem{assumption}[theorem]{Assumption}

\newtheorem{corollary}[theorem]{Corollary}
\newtheorem{definition}[theorem]{Definition}
\newtheorem{example}[theorem]{Example}
\newtheorem{examples}[theorem]{Examples}

\newtheorem{notation}[theorem]{Notation}
\newtheorem{openproblem}[theorem]{Open Problem}
\newtheorem{observation}[theorem]{Observation}
\newtheorem{remark}[theorem]{Remark}

\usepackage{newmacros}

\usepackage[colorlinks,citecolor=blue]{hyperref}
\begin{document}

\title[Sifted Colimits]{Sifted Colimits, Strongly Finitary Monads and Continuous Algebras}
\author{J.~Adámek}
\thanks{J.~Adámek and M.~Dostál acknowledge the support by
	the Grant Agency of the Czech Republic under the grant 22-02964S}
\address{Department of Mathematics, Faculty of Electrical Engineering, Czech Technical University in Prague, Czech Republic and Institute for Theoretical Computer Science, Technical University Braunschweig, Germany}
\email{j.adamek@tu-bs.de}
\author{M.~Dostál}
\author{J.~Velebil}
\address{Department of Mathematics, Faculty of Electrical Engineering, Czech Technical University
	in Prague, Czech Republic}
\email{$\{$dostamat,velebil$\}$@fel.cvut.cz}

\begin{abstract}
We characterize strongly finitary monads on categories $\Pos$, $\CPO$ and $\DCPO$ as precisely those preserving sifted colimits.
Or, equivalently, enriched finitary monads preserving reflexive coinserters.
We study sifted colimits in general enriched categories.

For $\CPO$ and $\DCPO$ we characterize varieties of continuous algebras as precisely the monadic categories for strongly finitary monads.
\end{abstract}

\keywords{Sifted colimit, monad, continuous algebra}

\maketitle

\section{Introduction}

Sifted colimits in ordinary categories, which are essentially combinations of filtered colimits and reflexive coequalizers, play a fundamental role in categorical algebra as demonstrated in~\cite{adamek+rosicky+vitale}.
For example, varieties of finitary algebras are precisely the free completions of duals of algebraic theories under sifted colimits.
For ordered algebras varieties, considered as enriched categories over $\Pos$ (the cartesian closed category of posets) are precisely the free completions of duals of enriched algebraic theories under enriched sifted theories~\cite{AR23}.
Sifted colimits are (as we prove below in Theorem~\ref{T:free} and~\ref{T:preser}) essentially combinations of filtered colimits and reflexive coinserters.

In the present paper we study sifted colimits in general enriched categories.
A weight $W$ is called sifted if colimits of diagrams in the base category weighted by $W$ commute with finite products.
Example: if the base category is cartesian closed, filtered colimits are sifted.
We characterize the free completion $\Sind \K$ of an enriched category $\K$ under sifted colimits.
For example: if the base category is $\Pos$, then we prove
\[
\Sind \K = \Ind (\Rci \K)
\]
where $\Rci \K$ is the free completion under reflexive coinserters, and $\Ind$ denotes, as usual, the free completion under filtered colimits (Theorem~\ref{T:free}).

We apply our results, besides $\Pos$, to the categories $\DCPO$ of directed-complete posets (where directed subsets have joins, and morphisms are continuous maps: monotone maps preserving directed joins) and $\CPO$ of $\omega$-complete posets (where $\omega$-chains have joins and morphisms, also called continuous maps, preserve $\omega$-joins).
For all these cartesian closed base-categories $\V$ we prove that an enriched functor between cocomplete $\V$-categories preserves sifted colimits iff it is strongly finitary.
(This concept, introduced by Kelly and Lack~\cite{kelly+lack:strongly-finitary}, means that the functor in question is an enriched left Kan extension of its restriction to finite sets.)
See Example~\ref{E:strong} and Propositions~\ref{P:CPO} and~\ref{P:DCPO}.

We study varieties of continuous algebras, which are complete posets endowed with continuous operations of a given signature.
They play an important role in the semantics of programs.
Varieties, i.e.\ equational classes, were studied intensively since 1970's, see e.g.~\cite{ADJ}, \cite{N} or~\cite{anr}.
Equational presentations apply \emph{extended terms}: besides variables and composite terms $t = \sigma(t_0,\dots,t_{n-1})$ for $n$-ary symbols $\sigma$ (and terms $t_i$) we also admit formal joins of terms.
A \emph{variety of continuous algebras} is a class presented by a set of equations between extended terms.
We study two variants: varieties over $\CPO$ and over $\DCPO$.
We call algebras over $\CPO$ \emph{continuous} and those over $\DCPO$ \emph{$\Delta$-continuous}.

Every variety $\Vvar$ of continuous algebras has free algebras, thus it generates a free-algebra monad $\Tmon_\Vvar$ on $\CPO$.
Moreover, $\Vvar$ is isomorphic to the category $\CPO^{\Tmon_\Vvar}$ of monadic algebras for $\Tmon_\Vvar$.
A question arises: which monads are of the form $\Tmon_\Vvar$?
We prove that each monad $\Tmon_\Vvar$ is strongly finitary (Theorem~\ref{T:necessity}) and, conversely, every strongly finitary monad on $\CPO$ is proved to be the free-algebra monad of a variety (Theorem~\ref{T:suff}).
We conclude that varieties of continuous algebras and strongly finitary monads on $\CPO$ bijectively correspond.
More precisely: the category of varieties is dually equivalent to the category of strongly finitary monads (Corollary~\ref{C:main}).
Analogously, varieties of $\Delta$-continuous algebras bijectively correspond to strongly finitary monads on $\DCPO$.

\subsection*{Related Work}
Sifted colimits in ordinary categories were introduced by Lair~\cite{La} (who called them `tamisante') and studied intensively e.g.\ in~\cite{adamek+rosicky:sifted}, \cite{ARV10}, \cite{chen}, and~\cite{Jo}.
In enriched categories they were introduced by the unpublished preprint~\cite{DV} and for the special case of the base category $\Cat$ (small categories) they were studied by Bourke~\cite{bourke:thesis}.
Properties of sifted colimits, and their relationship to strongly finitary functors, have not been investigated in the enriched context so far.

Our results on continuous algebras are closely related to a number of results characterizing monads on a category $\C$ corresponding to varieties of algebras in $\C$:
\begin{enumerate}
	\item For $\C = \Set$ this is a classical result due to Linton~\cite{Li}: varieties correspond to finitary monads on $\Set$, see e.g.~\cite{maclane:cwm}, Theorem~VI.8.1.
	Moreover, a set functor is finitary iff it is strongly finitary, i.e.\ a left Kan extension of its restriction to finite sets.
	
	\item For $\C = \Pos$ it was Kurz and Velebil~\cite{kurz-velebil:exactness} who proved that varieties of ordered algebras (classes presented by inequations) correspond to strongly finitary monads.
	We have presented a simplified proof in~\cite{adv:ordered-algebras}.

	\item The idea of using extended terms (Definition~\ref{D:ext}) stems from~\cite{anr}.
	There Birkhoff's Variety Theorem was proved: varieties are precisely the HSP classes.
	However, the terms used in op.~cit.\ are a bit more general than those we introduce below: for allowing to form a term $t = \bigvee_{k \in \Nat} t_k$ we request that the number of variables in all the terms $t_k$ be finite.
	
	\item Strongly finitary monads on $\CPO$ are also studied by Jiří Rosický~\cite{JR23}.
	He proves a bijective correspondence to varieties of algebras, but his syntax is different from ours.
	
	\item Since $\CPO$ is a locally $\aleph_1$-presentable category, another bijective correspondence to varieties is presented by Bourke and Garner~\cite{bourke-garner:monads-and-theories}.
	
\end{enumerate}

{\textbf{Acknowledgement.}}
The authors are grateful to Jiří Rosický for interesting discussions.

\section{Ordintary Sifted Colimits}

We present a short summary of results on sifted colimits in ordinary categories (for comparison with the weighted sifted colimits studied in later sections).
These colimits were called `tamisante' by Lair~\cite{La}, the name `sifted' stems from~\cite{adamek+rosicky:sifted}.

\begin{definition}
\label{D:sif}
A small category $\D$ is called \emph{sifted} if colimits of diagrams over it in $\Set$ commute with finite products.
\end{definition}

\begin{example}[\cite{adamek+rosicky:sifted}]
\label{E:copr}
Every small category with finite coproducts is sifted.
\end{example}

\begin{proposition}[\cite{gu} or \cite{adamek+rosicky:sifted}, Theorem~2.15]
\label{P:delta}
A small category $\D$ is sifted iff $\D \neq \emptyset$ and the diagonal $\diag: \D \to \D \times \D$ is a final functor.
\end{proposition}

\begin{definition}
A \emph{sifted colimit} in a category is a colimit of a diagram with a sifted domain.
\end{definition}

\begin{examples}
\label{E:sif}
\phantom{phantom}
\begin{enumerate}

	\item  Filtered colimits are sifted.
	Indeed, analogously to Definition~\ref{D:sif}, a possible definition of a filtered category $\D$ is that colimits of diagrams over $\D$ in $\Set$ commute with finite limits.
	
	\item \emph{Reflexive coequalizers} are sifted colimits.
	Recall that a parallel pair is called \emph{reflexive} if it consists of split epimorphisms with a joint splitting.
	A reflexive coequalizer is a colimit of a diagram whose domain os $\Delta_1^\op$ for the following truncation $\Delta_1$ of $\Delta$ (the category of positive finite ordinals $n = \{ 0, \dots, n-1 \}$ and monotone maps):
	
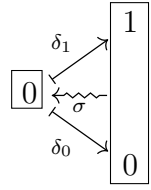
\[
	\tikz[%remember picture, 
	overlay]{
		\filldraw[fill=white,draw=black] (1.35,-1.2) rectangle (1.85,1.2);
		\filldraw[fill=white,draw=black] (0.05,-0.2) rectangle (0.45,0.3);
	}
	\begin{tikzcd}[row sep=small, column sep=tiny]
		&&&
		1
%		\arrow[dd, no head]
		&& \\
		0
		\arrow[urrr, mapsto, "\delta_1"]
		\arrow[drrr, mapsto, "\delta_0", swap]
		&&&
		\phantom{x} \arrow[lll, rightsquigarrow, "\sigma"]
		&&%\; \phi Y
		\\
%		\phantom{\phi X}
		&&& 0 &&
	\end{tikzcd}
	\]
\end{enumerate}
\end{examples}

\begin{remark}
\label{R:coeq}
We can state (cum grano solis) that
\begin{center}
sifted colimits = filtered colimits + reflexive coequalizers.
\end{center}
For example, if $\K$ is a category with finite coproducts, then the following are equivalent:
\begin{enumerate}
	
	\item $\K$ has filtered colimits and reflexive coequalizers.
	
	\item $\K$ has sifted colimits.
	
	\item $\K$ is cocomplete.
	
\end{enumerate}
Indeed (1) implies that $\K$ has coproducts (using filtered colimits).
Thus (3) follows: the famous construction of colimits via coproducts and coequalizers uses, in fact, reflexive coequalizers.
A more important support of the above slogan is the following
\end{remark}

\begin{theorem}[\cite{ARV10}, Theorem 2.1]
\label{T:coeq}
Let $\K$ and $\LL$ be cocomplete categories.
A functor $F: \K\to \LL$ preserves sifted colimits iff it preserves filtered colimits and reflexive coequalizers.
\end{theorem}

\begin{notation}
\phantom{phantom}
\begin{enumerate}

	\item $\Sind \K$ denotes the free completion of a category $\K$ under sifted colimits.
	(Analogous to the free completion $\Ind \K$ under filtered colimits.)
	
	\item $\Rec \K$ denotes the free completion of $\K$ under reflexive coequalizers.
	
\end{enumerate}
\end{notation}

\begin{theorem}[\cite{adamek+rosicky:sifted}, Corollary 2.7 and 2.8]
$\Sind \K$ is the category of all presheaves that are sifted colimits of representables in $[\K^\op,\Set]$.
If $\K$ has finite coproducts, then $\Sind \K$ is the category of all presheaves preserving finite products.
\end{theorem}

\begin{corollary}[\cite{adamek+rosicky:sifted}, Theorem 3.10]
Varieties of finitary algebras are, up to equivalence, precisely the categories $\Sind \K$, where $\K$ is a small category with finite coproducts.
\end{corollary}

One of the most important supports of the slogan in Remark~\ref{R:coeq} is the following

\begin{theorem}
\label{T:pb}
If a category $\K$ has either pullbacks or finite coproducts, then
\[
\Sind \K = \Ind(\Rec \K).
\]
\end{theorem}

The case of pullbacks was proved by Chen~\cite{chen}, Theorem~9.1, for finite coproducts see~\cite{adamek+rosicky:sifted}, Corollary~2.8.

\begin{example}
\phantom{phantom}
\begin{enumerate}

	\item For the category $\Set_\fin$ of finite sets we have
	\[
	\Sind \Set_\fin = \Ind \Set_\fin = \Set.
	\]
	Analogously let $\Vect$ be the category of vector spaces over a given field and $\Vect_\fin$ the full subcategory of finite-dimensional spaces.
	Then
	\[
	\Sind \Vect_\fin = \Ind \Vect_\fin = \Vect.
	\]
	
	\item For the category $\Ab$ of abelian groups if $\Ab_\fg$ denotes the subcategory of free finitely generated groups, then
	\[
	\Sind \Ab_\fg = \Ab.
	\]
	In contrast, $\Ind \Ab_\fg$ is the subcategory of all free abelian groups.
	
\end{enumerate}
\end{example}

\section{Sifted Weighted Colimits}

A natural generalization of sifted colimits to $\V$-categories, where $\V$ is a symmetric monoidal closed category, was presented by~\cite{DV}.
Here we show that under certain assumptions sifted colimits are just a combination of filtered colimits and reflexive coinserters.

\begin{assumption}
In this section $\V = (\V_\ordi,\tensor,I)$ denotes a symmetric monoidal closed category with finite products.
\end{assumption}

Let us recall the concept of a weighted colimit in a $\V$-category $\K$~\cite{BK}.
Given a diagram $D: \D \to \K$ and a weight $W: \D^\op \to \V$, both $\V$-functors, a weighted colimit is an object
\[
C = \Colim{W}{D}
\]
together with isomorphisms
\[
\phi_X : \K(C,X) \xrightarrow{\sim} [\D^\op, \V](W,\K(D\blank,X))
\]
natural in $X \in \K$.
The unit of the colimit is the natural transformation
\[
u = \phi_C(\id_C) : W \to \K(D\blank,C).
\]

\begin{examples}
We concentrate on four categories $\V$.
They are all cartesian closed, and in each case $\V$-natural transformations between $\V$-functors are just ordinary natural transformations between the ordinary underlying functors.
\begin{enumerate}

	\item $\Pos$: posets and monotone maps.
	A $\Pos$-enriched category carries partial orders on hom-sets making composition monotone.
	A $\Pos$-enriched functor $F: \K\to \LL$ is locally monotone: $f \sqleq g$ in $\K(A,B)$ implies $Ff \sqleq Fg$ in $\LL(FA,FB)$.

	\item $\CPO$: cpos and continuous maps. See Appendix A.
	
	\item $\DCPO$: dcpos and $\Delta$-continuous maps. See Appendix A.
	
	\item $\Cat$: small categories and functors.
	A $\Cat$-enriched category carries a category structure on hom-sets making composition a functor.
	
\end{enumerate}
\end{examples}

We denote by $\Pos_\fin$ the full subcategory of finite posets in $\Pos$, analogously for $\CPO_\fin$ and $\DCPO_\fin$.

\begin{definition}
\phantom{phantom}
\begin{enumerate}
	
	\item A weight $W: \D^\op \to \V$ is \emph{sifted} if colimits of diagrams in $\V$ weighted by $W$ commute with finite products.
	
	\item Let $\K$ be a $\V$-category.
	A \emph{sifted colimit} in $\K$ is a colimit weighted by a sifted weight.
	
\end{enumerate}
\end{definition}

\begin{example}
\label{E:sift}
Analogously to Example~\ref{E:copr}, if $\D$ has finite conical coproducts, every weight preserving finite coproducts is sifted (\cite{kelly+lack:strongly-finitary}, Lemma~2.3).
\end{example}

\begin{remark}
The following proposition was proved by Peter Johnstone in his PhD thesis~\cite{PJ}.
We provide a short proof for the convenience of the reader.
\end{remark}

\begin{proposition}
\label{P:fil}
For a cartesian closed category $\V$ filtered colimits are sifted.
\end{proposition}

\begin{proof}
\phantom{phantom}
\begin{enumerate}
	
	\item For an arbitrary small $\V$-category $\D$ we first observe that given diagrams $D_1,D_2: \D \to \V$ and forming
	\[
	D_1 * D_2: \D \times \D \to \V, \quad (x_1,x_2) \mapsto (D_1 x_1) \times (D_2 x_2)
	\]
	then
	\[
	\colim (D_1 * D_2) = (\colim D_1) \times (\colim D_2).
	\]
	More precisely: let $(c_x^1)$ be the colimit cocone of $D_1$ and $(c_y^2)$ that of $D_2$, then $D_1 * D_2$ has the colimit cocone $c_x^1 \times c_y^2: (D_1 x) \times (D_2 y) \to (\colim D_1) \times (\colim D_2)$.
	Indeed, for every $x \in \ob D_1$ we have
	\[
	(D_1 x) \times C_2 = \colim_y (D_1 x \times D_2 y)
	\]
	since $(D_1 x) \times \blank$ is a left adjoint.
	Thus, using that iterated colimits are simultaneous colimits (\cite{maclane:cwm}, Proposition p.\ 230), we get
	\begin{align*}
	\colim D_1 * D_2 & = \colim_{(x,y)} ((D_1 x) \times (D_2 y)) \\
	& \cong \colim_x \colim_y (D_1 x) \times (D_2 y) \\
	&  = \colim_x (D_1 x) \times C_2 \\
	& \cong C_1 \times C_2
	\end{align*}
	since $\blank \times C_2$ is a left adjoint.
	
	\item Let $\D$ be a small filtered category.
	Then $\diag: \D \to \D \times \D$ is final (Proposition~\ref{P:delta}).
	Given diagrams $D_1,D_2: \D \to \V$, form $D_1 * D_2$ as in (1), so that the triangle below commutes:
	\[
	\begin{tikzcd}
	{\D} \ar[rr, "\diag"] \ar[rd, swap, "D_1 \times D_2"] & & {\D\times \D} \ar[ld, "D_1 * D_2"] \\
	& {\V} &
	\end{tikzcd}
	\]
	From (1) we get, since $\colim D \cong \colim (D_1 \times D_2)$, the desired canonical isomorphism
	\[
	\colim (D_1 \times D_2) \cong (\colim D_1) \times (\colim D_2).
	\]
	
\end{enumerate}
\end{proof}

Using an analogous argument, we see that in every cartesian closed category reflexive coequalizers are sifted colimits.
In categories enriched over $\Pos$ (or over a basic subcategory, see below) we prove that reflexive coinserters are sifted colimits.
We recall the concept of a coinserter now and show below that it is a weighted colimit:

\begin{example}
\label{E:coins}
In the category $\Pos$ consider a parallel pair\footnote{We use indices $0$ and $1$ to indicate that $f_0$ comes first and $f_1$ second.
Thus $(f_0,f_1)$ is an \emph{ordered} pair; however $f_0 \sqleq f_1$ is \emph{not} assumed.}
$f_0,f_1: A \to B$.
Its \emph{coinserter} is a universal morphism $c: B \to C$ with respect to $c \comp f_0 \sqleq c \comp f_1$.
That is:
\begin{enumerate}

	\item Every morphism $c': B \to C'$ with $c' \comp f_0 \sqleq c' \comp f_1$ factorizes through $c$.

	\item Given $u,v: C \to D$ with $u \comp c \sqleq v \comp c$, it follows that $u \sqleq v$.
	
\end{enumerate}
A coinserter can be constructed as follows: consider the category of pre-orders.
Every pre-order $(X,\sqleq)$ has a posetal reflection: the quotient modulo the equivalence $\sqleq \cap \sqsupseteq$.
Let $\preceq$ be the least preorder on $|B|$ which contains both the order of $B$ and the relation
\[
\{ (f_0(a),f_1(a)) \mid a \in A \}.
\]
Then the coinserter $c: B \to C$ is the posetal reflecion of $(B,\preceq)$.
\end{example}

\begin{remark}
\label{R:id}
Every poset $P$ is a coinserter of a parallel pair between two sets (= discrete posets).
	Indeed, if $|P|$ denotes the underlying set and $R \subseteq |P| \times |P|$ the order-relation with projections $\pi_0,\pi_1: R \to |P|$, then the following is a coinserter:
	\[
	\begin{tikzcd}[column sep=3em]
		R \ar[r, bend left, "\pi_1"above] \ar[r, bend right, swap, "\pi_0"below] & {|P|} \ar[r, "\id"] & P
	\end{tikzcd}
	\]
%	\item Coinserters are weighted colimits, as we make explicit next.
\end{remark}

\begin{example}
\label{E:N}
Consider the following $\omega$-chain of the posets $C_k = \{0, 1, \dots, k\}$:
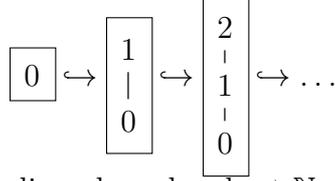
\[
\tikz[%remember picture, 
overlay]{
		\filldraw[fill=white,draw=black] (0,-0.2) rectangle (0.6,0.5);
	}
\begin{tikzcd}[row sep=tiny, column sep=tiny]
	0
\end{tikzcd}
\hookrightarrow
\tikz[%remember picture, 
overlay]{
		\filldraw[fill=white,draw=black] (0,-0.9) rectangle (0.6,0.9);
	}
\begin{tikzcd}[row sep=tiny, column sep=tiny]
	1
	\arrow[dd, no head]
	\\
	\\
	0
\end{tikzcd}
\hookrightarrow
\tikz[%remember picture, 
overlay]{
		\filldraw[fill=white,draw=black] (0,-1.2) rectangle (0.6,1.2);
	}
\begin{tikzcd}[row sep=tiny, column sep=tiny]
	2
	\arrow[d, no head]
	\\
	1
	\arrow[d, no head]
	\\
	0
\end{tikzcd}
\hookrightarrow \dots
\]
Its colimit in $\Pos$ is the linearly ordered set $\Nat$.
In contrast, the colimit in $\CPO$ is
\[
\Nat^\top = \Nat \cup \{ \top \}
\]
obtained from the chain of natural numbers by adding a top element $\top$ to it.
Thus $\Nat^\top$ is contained in the closure of $\Set_\fin$ under reflexive coinserters and $\omega$-colimits (using the preceding remark).
\end{example}

\begin{definition}
\label{D:basic}
A subcategory $\V$ of $\Pos$, not necessarily full, is \emph{basic} if it is cartesian closed, closed under limits, contains $\Pos_\fin$ as a full subcategory, and contains all morphisms of the form $1 \to X$ for all $X \in \ob \V$.
\end{definition}

\begin{example}
$\CPO$ and $\DCPO$ are basic subcategories of $\Pos$.
\end{example}

\begin{definition}
\label{D:coins}
Let $\V$ be a basic subcategory of $\Pos$.
Then the \emph{coinserter} of a parallel pair $f_0,f_1: A \to B$ is the weighted colimit of the following diagram
\[
\tikz[%remember picture, 
overlay]{
	\filldraw[fill=white,draw=black] (0.0,-0.2) rectangle (2.3,1.2);
	\filldraw[fill=white,draw=black] (3.7,-0.2) rectangle (6.0,1.2);
}
\begin{tikzcd}[column sep=scriptsize,row sep=small]
	{a \bullet} & {\bullet b} && A & B \\
	& \D &&& \V
	\arrow["{f_1}", curve={height=-6pt}, from=1-4, to=1-5]
	\arrow["{f_0}"', curve={height=6pt}, from=1-4, to=1-5]
	\arrow["{\varphi_0}"', curve={height=6pt}, from=1-1, to=1-2]
	\arrow["{\varphi_1}", curve={height=-6pt}, from=1-1, to=1-2]
	\arrow["D", shorten <=10pt, shorten >=10pt, from=1-2, to=1-4]
\end{tikzcd}\
\]
with hom-objects of $\D$ discrete and with the following weight $W_0$
\[
\tikz[%remember picture, 
overlay]{
	\filldraw[fill=white,draw=black] (1.3,-1.2) rectangle (1.9,1.3);
	\filldraw[fill=white,draw=black] (2.5,-0.2) rectangle (2.9,0.3);
}
\begin{tikzcd}[row sep=small, column sep=tiny]
	&
	1
	\arrow[dd, no head]
	&&&  \\
	W_0 a &
	&& *  %\; \phi Y
	\arrow[ull, mapsto, "W_0 \phi_1", swap]
	\arrow[dll, mapsto, "W_0 \phi_0"]
	& W_0 b
	\\
	\phantom{\phi X} & 0 &&&
\end{tikzcd}
\]
(where $0 \sqleq 1$).
We speak about \emph{reflexive coinserters} in case the pair $f_0,f_1$ is reflexive.
\end{definition}

Explicitly, a coinserter in $\V$ is given by a morphism $c: B \to A$ universal with respect to $c \comp f_0 \sqleq c \comp f_1$.
Indeed, the unit $u: W \to [D\blank,C]$ of the colimit $C = \Colim{W}{D}$ has the component $u_b$ given by $c: B \to C$ and the component $u_a$ given by $c \comp f_0$ and $c \comp f_1$.
Thus $W \phi_0 \sqleq W \phi_1$ yields $c \comp f_0 \sqleq c \comp f_1$.
The universality of $c$ follows from the definition of weighted colimit.

\begin{theorem}
\label{T:rci}
Reflexive coinserters are sifted colimits in every basic subcategory $\V$ of $\Pos$.
\end{theorem}

\begin{proof}
It is our task, given coinserters with splittings $\delta$ and $\delta'$ as follows:
\[
\begin{tikzcd}
	{A}
	\arrow[r, shift left, "f_1"]
	\arrow[r, shift right, swap, "f_0"]
	&
	{B}
	\arrow[r, "c"]
	\arrow[shiftarr={yshift=5ex}, l, swap, "\delta"] 
	&
	C
\end{tikzcd}
\quad
\begin{tikzcd}
	{A}
	\arrow[r, shift left, "f_1'"]
	\arrow[r, shift right, swap, "f_0'"]
	&
	{B}
	\arrow[r, "c'"]
	\arrow[shiftarr={yshift=5ex}, l, swap, "\delta'"] 
	&
	C
\end{tikzcd}
\]
to prove that $c \times c'$ is a coinserter of $f_0 \times f_0'$ and $f_1 \times f_1'$.
Thus reflexive coinserters commute with binary products.
The statement for empty product is trivial.
\begin{enumerate}[label=(\alph*)]
	
	\item Given a morphism $d : B \times B' \to D$ with
	\begin{equation}
		\tag{q1}
		d \comp (f_0 \times f_0') \sqleq d \comp (f_1 \times f_1')
	\end{equation}
	we prove that it factorizes through $c \times c'$.
	Precomposing with $\delta \times A'$ we get $d \comp (B \times f_0') \sqleq d \comp (B \times f_1')$.
	Therefore, the adjoint transpose $\wh{d} : B' \to [B,D]$ fulfils
	\[
	\wh{d} \comp f_0' \sqleq \wh{d} \comp f_1' : B' \to [B,D].
	\]
	By the universal property of $c'$ we thus get a morphism $g : B \times C' \to D$ such that $\wh{g}$ is a factorization of $\wh{d}$ through $c'$:
	\[
	\begin{tikzcd}
		B' \ar[r, "c'"] \ar[rd, swap, "\wh{d}"] & C' \ar[d, "\wh{g}"] \\
		& {[B,D]}
	\end{tikzcd}
	\]
	In other words, the following lower triangle commutes:
	\begin{equation}
		\label{q2}
		\tag{q2}
		\begin{tikzcd}
			{A \times B'} \ar[r, "A \times c'"] \ar[d, shift left, "f_1 \times B'"] \ar[d, shift right, swap, "f_0 \times B'"] & {B \times C'} \ar[d, shift left, "f_1 \times C'"] \ar[d, shift right, swap, "f_0 \times C'"] \\
			{B \times B'} \ar[r, "B \times c'"] \ar[rd, swap, "d"] & {B \times C'}  \ar[d, "g"] \\
			& D
		\end{tikzcd}
	\end{equation}
	Now precomposing (q1) with $A \times \delta'$ yields in (the square above) that
	\[
	d \comp (f_0 \times B') \sqleq d \comp (f_1 \times B').
	\]
	Thus the above diagram proves
	\[
	g \comp (f_0 \times C') \comp (A \times c') \sqleq g \comp (f_1 \times C') \comp (A \times c').
	\]
	From (1) we know that $A \times c'$ is the coinserter of $A \times f_0', A \times f_1'$.
	Using the universal property of this coinserter, we get
	\[
	g \comp (f_0 \times C') \sqleq g \comp (f_1 \times C').
	\]
	From (1) we also know that $f_0 \times C', f_1 \times C'$ have the coinserter $c \times C'$, and thus $g$ factorizes as follows
	\[
	\begin{tikzcd}
		{B \times C'} \ar[r, "c \times C'"] \ar[rd, swap, "g"] & {C \times C'}  \ar[d, "h"] \\
		& D
	\end{tikzcd}
	\]
	Combining this with the lower triangle in~\eqref{q2} yields
	\[
	d = g \comp (B \times c') = h \comp (c \times c')
	\]
	as desired.
	
	\item Let $u, v : C \times C' \to X$ fulfil
	\[
	u \comp (c \times c') \sqleq v \comp (c \times c').
	\]
	We prove $u \sqleq v$.
	We can rewrite the above inequality as follows
	\[
	[u \comp (c \times B')] \comp (C \times c') \sqleq [v \comp (c \times B')] \comp (C \times c').
	\]
	Denote the adjoint transposes of $u \comp (c \times B')$ and $v \comp (c \times B')$ by $\wt{u}, \wt{v} : B' \to [C,X]$, respectively.
	The adjoint transpose of the above inequation yields
	\[
	\wt{u} \comp c' \sqleq \wt{v} \comp c'.
	\]
	The universal property of $c'$ thus implies $\wt{u} \sqleq \wt{v}$, which (by inverse to the adjoint transpose) proves
	\[
	u \comp (c \times B') \sqleq v \comp (c \times B').
	\]
	Since by (1) the morphism $c \times B'$ is a coinserter, we get $u \sqleq v$.
	
\end{enumerate}
\end{proof}

Analogously to Remark~\ref{R:coeq} we can state, for basic subcategories of $\Pos$, that
\begin{center}
sifted colimits = filtered colimits + reflexive coinserters.
\end{center}
For example the following proposition supports this, more important support is presented in Section 4.

\begin{proposition}
Let $\V$ be a basic subcategory of $\Pos$.
If $\K$ is a $\V$-category with finite conical coproducts, then equivalent are:
\begin{enumerate}
	
	\item $\K$ has filtered colimits and reflexive coinserters.
	
	\item $\K$ has sifted colimits.
	
	\item $\K$ is cocomplete.
	
\end{enumerate}
\end{proposition}

\begin{proof}
Due to Theorem~\ref{T:rci} and Proposition~\ref{P:fil} we just need to verify that (1) implies (3).

\begin{enumerate}[label=(\alph*)]
	
	\item $\K$ has colimits.
	Indeed, it has conical coproducts since they are filtered colimits of finite ones.
	It also has reflexive coequalizers (Example~\ref{E:sif}):
	given a reflexive pair $u,v: A \to B$ the pair $f_0 = [u,v] : A + A \to B$ and $f_1 = [v,u]: A + A \to B$ is also reflexive, and a coinserter of $f_0,f_1$ is precisely a coequalizer of $u$ and $v$.
	Thus, $\K$ has conical colimits: the construction of colimits via coproducts and coequalizers (\cite{maclane:cwm}, Theorem~V.2.1) uses reflexive coequalizers only.
	
	\item $\K$ has tensors: given a poset $P$, to verify that a tensor $P \tensor X$ exists for each $X \in \ob \V$ we can restric ourselves to $P$ finite and use the fact that if $P = \colim_{i \in I} P_i$ is a filtered colimit in $\Pos$ with $P_i$ finite, then $P \tensor X = \colim_{i \in I} P_i \tensor X$ (whenever all tensors $P_i \tensor X$ exist in $\V$).
	
	For $P$ finite and discrete $P \tensor X = \coprod_{|P|} X$ is a copower. For $P$ non-discrete use Remark~\ref{R:id}: we get the tensor $P \tensor X$ by forming the reflexive coinserter of $\pi_0 \tensor \id, \pi_1 \tensor \id: R \tensor X \to |P| \tensor X$: since $\blank \tensor X$ preserves weighted colimits, that coinserter has the form $c \tensor \id : |P| \tensor X \to P \tensor X$.
	
	\item  $\K$ is cocomplete since $\V_\ordi$ is complete and cocomplete: this follows from (the dual of) \cite{borceux:second}, Theorem~6.6.14.
	
\end{enumerate}
\end{proof}

\begin{examples}
\label{E:Bourke}
\phantom{phantom}
\begin{enumerate}

	\item For $\CPO$-enriched or $\DCPO$-enriched categories with finite products the existence of filtered colimits and reflexive coinserters implies cocompleteness.
	
	\item For the cartesian closed category $\Cat$ of small categories and functors Bourke~\cite{bourke:thesis} studied sifted colimits.
	The role of reflexive coinserters is played by the colimits weighted by the embedding
	\[
	W_2: \Delta_2^\op \hookrightarrow \Cat.
	\]
	Here $\Delta_2$ is the truncation of $\Delta$ (Example~\ref{E:sif} (2)) to $0$,$1$,$2$ and $W1$ and $W2$ are the categories given by the linear order $0 < 1 < 2$.
	Bourke proved e.g.\ that an enriched endofunctor of $\Cat$ preserves sifted colimits iff it preserves
	\begin{enumerate}
		\item Filtered colimits.
		\item Codescent objects of strict reflexive data, which means precisely colimits weighted by $W_2$ (\cite{bourke:thesis}, Corollary~8.45).
	\end{enumerate} 
	We shall see similar results for $\Pos$-enriched categories in Section 5.
\end{enumerate}
\end{examples}

\section{The completion $\Sind$}

A free completion of a $\V$-category $\K$ under sifted colimits is called $\Sind \K$.
We observe that it consists of presheaves that are sifted colimits of representables.
In case $\V = \Pos$, we also introduce the free completion $\Rci \K$ under reflexive coinserters and prove
\[
\Sind \K= \Ind (\Rci \K)
\]
for all enriched categories $\K$ with finite coproducts.
And we derive that an enriched functor between cocomplete categories preserves sifted colimits iff it preserves filtered colimits and reflexive coinserters.

Much of what we do with sifted weights is simply a special case of the work of Kelly and Schmitt~\cite{kelly+schmitt} that we shortly recall first.

\begin{assumption}
\label{ASS}
Throughout this and the next section $\V = (\V_\ordi,\tensor,I)$ denotes a closed symmetric monoidal category which is complete and cocomplete (has weighted limits and colimits).
Moreover, for the category $\Set_\fin$ of finite sets the functor
\[
K: \Set_\fin \to \V, \qquad KX = \coprod_X I,
\]
is assumed to be a full embedding.
\end{assumption}

All the categories $\Set$, $\Pos$, $\CPO$, $\DCPO$ and $\Cat$ satisfy these assumptions.

\begin{notation}[\cite{kelly+schmitt}]
Let $\Phi$ be a class of weights.
A $\V$-category is called \emph{$\Phi$-cocomplete} if it has colimits of diagrams weighted in $\Phi$.
A $\V$-functor is called \emph{$\Phi$-cocontinuous} if it preserves colimits weighted in $\Phi$.
The $\V$-category of all such functors from $\K$ to $\LL$ is denoted by
\[
\Phi\text{-}\Cocts(\K,\LL).
\]
\end{notation}

\begin{definition}[\cite{kelly+schmitt}]
A \emph{free completion} of a $\V$-category $\K$ under $\Phi$-colimits is a $\Phi$-cocomplete $\V$-category $\ol{\K}$ together with a $\V$-functor $E: \K \to \ol{\K}$ yielding an equivalence of categories by precomposition:
\[
(\blank) \comp E: \Phi\text{-}\Cocts(\ol{\K},\LL) \xrightarrow{\simeq} [\K,\LL].
\]
\end{definition}

It follows from~\cite{kelly+schmitt}, Proposition~4.1 that $E$ is always fully faithful.
Thus without loss of generality we can consider $\ol{\K}$ to contain $\K$ as a full subcategory (and $E$ to be the embedding).
Kelly and Schmitt give a characterization, for an arbitrary class $\Phi$ of weights, of free completions under $\Phi$-colimits:

\begin{proposition}[\cite{kelly+schmitt}, Proposition 4.3]
\label{P:general}
A $\V$-full embedding
\[
E: \K \hookrightarrow \ol{\K}
\]
is a free completion under $\Phi$-colimits iff
\begin{enumerate}

	\item $\ol{\K}$ is $\Phi$-cocomplete.
	
	\item The functor $\ol{\K}(X,\blank)$ is $\Phi$-cocontinuous for every $X \in \ob \K$.
	
	\item $\ol{\K}$ is an iterated closure of $\K$ under $\Phi$-colimits.
	
\end{enumerate}
\end{proposition}

\begin{remark}
The free completion of a $\V$-category of $\K$ under filtered (conical) colimits is usually denoted by $\Ind \K$.
It is well known that $\Ind \K$ can be described as the category of all presheaves in $[\K^\op,\V]$ which are filtered colimits of representables. Analogously:
\end{remark}

\begin{notation}
The free completion of a $\V$-category $\K$ under sifted colimits is denoted by
\[
\Sind \K.
\]
\end{notation}

\begin{proposition}
For every $\V$-category $\K$ we can describe $\Sind \K$ as the category of all presheaves in $[\K^\op,\V]$ that are sifted colimits of representables.
\end{proposition}

More detailed: the restricted Yoneda embedding $Y$ of $\K$ into the above subcategory of $[\K^\op,\V]$ has the universal property.
This follows from~\cite{albert+kelly}, Proposition~7.3, since the class $\Phi$ of all sifted weights is saturated (aka closed).
The saturation of $\Phi$ is a consequence of the fact that $\Phi$ is defined via commutation with certain limits: see~\cite{kelly+schmitt}, Proposition~5.4.

\begin{corollary}
\label{C:sind}
For a small $\V$-category $\K$ with finite conical coproducts we have
\[
\Sind \K = \text{all finite-product preserving presheaves.}
\]
\end{corollary}

Indeed, a presheaf $F: \K^\op \to \V$ preserves finite products iff it is a sifted colimit of representables.
The `if' direction is clear from the definition of a sifted weight and the fact that representables preserve finite products.
For the `only if' direction recall that $F$ is a sifted weight (Example~\ref{E:sift}).
Thus, the standard representation of $F$ as the colimit of $Y$ weighted by $F$
\[
F = \Colim{F}{Y}
\]
presents $F$ as a sifted colimit of representables.

\begin{remark}
\label{R:ind}
\phantom{phantom}
\begin{enumerate}

	\item Analogously, for a small $\V$-category $\K$ with finite conical limits we have
	\[
	\Ind \K = \text{all finite-limit preserving presheaves.}
	\]
	In fact, these are precisely the presheaves that are filtered colimits of representables.
	
	\item For large categories, we have to work with \emph{small presheaves}: those that are small colimits of representables.
	If a $\V$-category $\K$ has finite conical coproducts (or colimits) then $\Sind \K$ consists of small presheaves preserving finite products (and $\Ind \K$ of all small presheaves preserving finite limits).
	
\end{enumerate}
\end{remark}

In the rest of this section we work with poset-enriched categories: $\V = \Pos$.
We will prove that in case $\K$ has finite conical coproducts, its free completion under reflexive coinserters is actually a special completion under finite colimits.
From that we derive that functors preserving filtered colimits and reflexive coinserters are finitely cocontinuous.

\begin{notation}
\label{N:Rci}
The free completion of a poset-enriched category $\K$ under reflexive coinserters is denoted by
\[
\Rci \K.
\]
\end{notation}

\begin{definition}
\label{D:e}
\phantom{phantom}
\begin{enumerate}

	\item An object $X$ of $\V$ is \emph{element-finite}, shortly \emph{e-finite}, if it has finitely many global elements: $\V_\ordi(1,X)$ is a finite set.
	
	\item A weight $W: \D^\op \to \V$ is e-finite if $\ob \D$ is finite and all objects
	\[
	\D(d,d') \text{ and } Wd \text{ ($d$,$d'\in \ob \D$)}
	\]
	are e-finite.
	
	\item Colimits weighted by e-finite weights are called e-finite.
	An \emph{e-finitely cocomplete} $\V$-category is one that has e-finite colimits.
	Functors preserving e-finite colimits are called \emph{e-finitely cocontinuous}.
	
\end{enumerate}
\end{definition}

\begin{proposition}
\label{P:e}
A $\V$-category has e-finite colimits iff it has finite conical colimits and tensors with e-finite objects of $\V$.
\end{proposition}

\begin{proof}
\phantom{phantom}
\begin{enumerate}

	\item Necessity.
	It is clear that the weights for finite conical colimits are e-finite.
	Let $P$ be an e-finite object of $\V$.
	Tensors $P \tensor \blank$ are precisely colimits weighted by $W: \I^\op \to \V$ where $\I$ is the terminal $\V$-category and $W$ represents $P$ (\cite{kelly:book}, Section~1.3).
	Again, $W$ is clearly e-finite.
	
	\item Sufficiency.
	Let $\K$ be a $\V$-category satisfying the above condition.
	Then for every e-finite weight $W: \D^\op \to \K$ and every diagram $D: \D \to \K$ we can form the finite coproduct
	\[
	\coprod_{d',d''\in \D} \D(d',d'') \tensor (Wd' \tensor Wd'').
	\]
	Moreover, we obtain a canonical pair of morphisms
	\[
	\begin{tikzcd}
	{\coprod_{d',d''\in \D} \D(d',d'') \tensor (Wd' \tensor Wd'')}
	\arrow[r, shift left]
	\arrow[r, shift right]
	&
	{\coprod_{d \in \D} Wd \tensor Wd}
	\end{tikzcd}
	\]
	whose coequalizer (exists and) yields $\Colim{W}{D}$: see~\cite{kelly:book}, the dual of~(3.08).
\end{enumerate}
\end{proof}

\begin{corollary}
\label{C:e}
A $\V$-functor between e-finitely cocomplete $\V$-categories is e-finitely cocontinuous iff it preserves finite conical colimits and tensors with e-finite objects.
\end{corollary}

By an \emph{e-finite cocompletion} of a $\V$-category $\K$ is meant an e-finitely cocomplete category containing $\K$ as a full subcategory.
If, moreover, $\K$ is closed under finite coproducts, we call the cocompletion \emph{plus-conservative}.

\begin{definition}
\label{D:free}
Let $\K$ be a $\V$-category with finite conical coproducts.
A plus-conservative e-finite cocompletion $\ol{\K}$ is \emph{free} if, give an e-finitely cocomplete category $\LL$, the category
\[
\A_1: \text{ all e-finitely cocontinuous functors in } [\ol{\K},\LL]
\]
is equivalent, via domain-restriction, to the category
\[
\A_2: \text{ all finite-coproduct preserving functors in } [\K,\LL].
\]
\end{definition}

\begin{remark}
\label{R:free}
\phantom{phantom}
\begin{enumerate}
	
	\item It follows that every functor $F: \K \to \LL$ preserving finite coproducts has a unique e-finitely cocontinuous extension $\ol{F}: \ol{\K} \to \LL$.
	
	\item In case $\V$ is a basic subcategory of $\Pos$, conversely, the property in (1) implies that $\ol{\K}$ is a free plus-conservative e-finite cocompletion.
	To see this, recall that $\V$ contains all morphisms from $1$; thus enriched natural transformations are just the ordinary ones.
	What we need to prove is that every natural transformation $\tau: F \to G$, a morphism in $\A_2$, yields a unique natural transformation $\ol{\tau}: \ol{F} \to \ol{G}$ between the e-finitely cocontinuous extensions with $\ol{\tau}_K = \tau_K$ for all $K \in \K$.
	
	Apply the above extension property to the morphism-$\V$-category:
	\[
	\LL^\rightarrow = [\Two,\LL] \text{ where $\Two$ is a two-chain}
	\]
	in place of $\LL$.
	To give a functor $H: \K\to \LL^\rightarrow$ in $\A_2$ means precisely to give $\tau: F \to G$ as above.
	And the unique extension $\ol{H}: \ol{\K} \to \LL^\rightarrow$ is precisely the desired natural transformation $\ol{\tau}$.
	
\end{enumerate}
\end{remark}

\begin{theorem}
\label{T:Rci}
Let $\K$ be a poset-enriched category with finite conical coproducts.
Then $\Rci \K$ is a free e-finite plus-conservative cocompletion of $\K$.
\end{theorem}

\begin{proof}
\phantom{phantom}
\begin{enumerate}

	\item The initial object $0$ of $\K$ is initial in $\Rci \K$.
	We prove this by structural induction: to verify that every object $A$ of $\Rci \K$ has a unique morphism from $0$, we denote by $\A \subseteq \Rci \K$ the full subcategory of all objects with the desired property.
	Then we verify that
	\begin{enumerate}
		\item $\A$ contains $\K$.
		\item $\A$ is closed under reflexive coinserters in $\Rci \K$.
	\end{enumerate}
	This proves $\A = \Rci \K$, as desired.
	The statement (a) is obvious.
	To prove (b), consider a reflexive pair $f_0,f_1: A \to B$ in $\A$ and its coinserter in $\Rci \K$:
	\[
	\label{coins}
	\tag{coins}
	\begin{tikzcd}
		{A}
		\arrow[r, shift left, "f_1"]
		\arrow[r, shift right, swap, "f_0"]
		&
		{B}
		\arrow[r, "c"]
		&
		C
	\end{tikzcd}
	\]
	This coinserter is preserved by $\Rci \K(0,\blank)$ by Proposition~\ref{P:general}.
	Thus every morphism in $\Rci \K(0,C)$ factorizes through $c$.
	Since $B \in \A$, this factorization is unique.
	Hence, $\Rci \K(0,C)$ is a singleton set, as required.

	\item $\K$ is closed under conical binary coproducts in $\Rci \K$.
	Indeed, let a coproduct in $\K$ be given:
	\[
	K = K_1 + K_2 \text{ with injections } v_1, v_2.
	\]
	We proceed by structural induction.
	Let $\A$ be the full subcategory of $\Rci \K$ on all objects $X$ such that (i) for every pair $k_i: K_i \to X$ there is $k: K \to X$ with $k_i = k \comp v_i$ and (ii) given $k,l: K \to X$ with $k \comp v_i \sqleq l \comp v_i$ ($i=1,2$), then $k \sqleq l$.
	\begin{enumerate}[label=(\theenumi\alph*)]
	
		\item $\K \subseteq \A$ because the coproduct $K = K_1 + K_2$ is conical in $\K$.
		
		\item If the coinserter~\eqref{coins} fulfils $A,B \in \A$, we prove $C \in \A$.
		For that we use that $\Rci \K (K_i, \blank)$ preserve that coinserter for $i=1,2$ (Proposition~\ref{P:general}).
		
	\end{enumerate}
	For Item (i), use that $\Rci \K (K_i, C)$ is epic: coinserters in $\Pos$, as described in Example~\ref{E:coins}, are surjective.
	Thus there exist morphisms $k_i': K_i \to B$ with $k_i = c \comp k_i'$ ($i=1,2$).
	Since $B \in \A$, we have a morphism $k': K \to B$ with $k_i' = k' \comp v_i$ ($i=1,2$).
	Then $k = c \comp k'$ fulfls $k_i = k \comp v_i$.
	
	For the proof of (ii) recall that the hom-functors
	\[
	F_j = \Rci \K(K_j, \blank): \Rci \K \to \Pos \qquad \text{($j=1,2$)}
	\]
	preserve reflexive coinserters.
	Thus $F_j c$ is the coinserter of $F_j f_0$ and $F_j f_1$.
	Since reflexive coinserters in $\Pos$ are sifted colimits, $F_1c \times F_2 c$ is the reflexive coinserter of $F_1 f_i \times F_2 f_i$ for $i=0,1$.
	Following the description of coinserters in $\Pos$ (Example~\ref{E:coins}), the morphism $F_1 c \times F_2 c$ is the posetal reflection of the least pre-order on $F_1 B \times F_2 B$ containing
	\begin{enumerate}
		
		\item the order of $F_1 B \times F_2 B$ and
		
		\item the relation of all pairs $(u_0,u_1)$ given by choosing an element $(p_1,p_2) \in F_1 A \times F_2 A$ and forming $u_0 = (f_0 \comp p_1, f_0 \comp p_2)$ and $u_1 = (f_1 \comp p_1, f_1 \comp p_2)$
		\[
		\begin{tikzcd}
		{K_1} \ar[rd, "p_1"] & {K_2} \ar[d, swap, "p_2"] & \\
		& A \ar[r, shift left, "f_1"] \ar[r, swap, shift right, "f_0"] & B
		\end{tikzcd}
		\]
		As in Item~(i), we can find, for our morphisms $k,l: K \to C$, morphisms $k',l': K \to B$ with $k = c \comp k'$ and $c \comp l'$.
		The given assumption $k \comp v_j \sqleq l \comp v_j$ means that in $F_1 C \times F_2 C$ we have
		\[
		(c \comp k' \comp v_1, c \comp k' \comp v_2) \sqleq (c \comp l' \comp v_1, c \comp l' \comp v_2)
		\]
		In the proof of $k \sqleq l$ we can thus restrict ourselves to two special cases yielding the above inequality in $F_1 C \times F_2 C$:
		\begin{enumerate}[label=(\greek*)]
			\item In $F_1 B \times F_2 B$ we have $(k' \comp v_1, k' \comp v_2) \sqleq (l' \comp v_1, l' \comp v_2)$.
			\item There is $(p_1,p_2) \in F_1 A \times F_2 A$ with
			\[
			k' \comp v_j = f_0 \comp p_j \text{ and } l' \comp v_j = f_1 \comp p_j \text{ ($j=1,2$).}
			\]
		\end{enumerate}
		In case ($\alpha$), we use $B \in \A$: since $k' \comp v_j = l' \comp v_j$ for $j=1,2$, we conclude $k' \sqleq l'$.
		Therefore $k = c \comp k' \sqleq c \comp l' = l$.
		In case ($\beta$) we also use $A \in \A$: there is $p: K \to A$ with $p_j = p \comp v_j$ ($j=1,2$).
		From $B \in \A$ and the equalities
		\[
		k' \comp v_j = f_0 \comp p \comp v_j \qquad \text{($j=1,2$)}
		\]
		we derive $k' = f_0 \comp p$.
		Analogously $l' = f_1 \comp p$.
		Thus $c \comp k' \sqleq c \comp l'$; i.e.\ $k \sqleq l$
		as desired.
		Thus $C \in \A$. 
		
	\end{enumerate}
	
	\item The category $\Rci \K$ has binary conical coproducts (thus all finite ones due to Item~(1)).
	To prove this, we fix an object $X$ of $\Rci \K$ and proceed by structural induction: let $\A$ be the full subcategory on all objects $Y$ with a conical coproduct $X + Y$ in $\Rci \K$. We first prove (b), then (a).
	
	\begin{enumerate}
		\item[(3b)] $\A$ is closed under reflexive coinserters.
		Suppose in the coinserter~\eqref{coins} we have $A,B \in \A$.
		We thus can form the (obviously reflexive) coinserter $\wt{c}$ of $\id + f_0$ and $\id + f_1$:
		\[
		\begin{tikzcd}
		A \ar[r, shift left, "f_1"] \ar[r, shift right, swap, "f_0"] \ar[d, swap, "inr"] & B \ar[r, "c"] \ar[d, "inr"] & C \ar[d, dotted, "u"] \\
		{X + A} \ar[r, shift left, "\id + f_1"] \ar[r, shift right, swap, "\id + f_0"] & {X + B} \ar[r, "\wt{c}"] & {\wt{C}} \\
		X \ar[u, "inl"] \ar[ur, swap, "inl"] & & 
		\end{tikzcd}
		\]
		We have $\wt{c} \comp (\id + f_0) \leq \wt{c} \comp (\id + f_1)$ which, precomposed by the coproduct injection $inr$, yields
		\[
		(\wt{c} \comp inr) \comp f_0 \leq (\wt{c} \comp inr) \comp f_1.
		\]
		Hence we obtain a unique $u$ making the square above commutative.
		We claim that in the underlying ordinary category $\K_\ordi$ the desired coproduct is
		\[
		X + C = \wt{C} \text{ with injections $\wt{c} \comp inl$ and $u$.}
		\]
		Indeed, consider an arbitrary pair of morphisms
		\[
		k: X \to D \text{ and } h : C \to D.
		\]
		The morphism
		\[
		l = [k, h \comp c]: X + B \to D
		\]
		clearly fulfils
		\[
		l \comp (\id + f_0) \leq l \comp (\id + f_1)
		\]
		and it thus factorizes through $\wt{c}$:
		\[
		\begin{tikzcd}
		{X + A} \ar[r, shift left, "\id + f_1"] \ar[r, shift right, swap, "\id + f_0"] & {X + B} \ar[r, "\wt{c}"] \ar[d, swap, "{[k, h \comp c]}"] & {\wt{C}} \ar[dl, dotted, "\wt{l}"] \\
		& D &
		\end{tikzcd}
		\]
		The morphism $\wt{l}$ is the desired factorization: we have
		\[
		\wt{l} \comp (\wt{c} \comp inl) = k.
		\]
		Moreover,
		\[
		\ol{l} \comp u = h
		\]
		since the coinserter $c$ is epic:
		\[
		(\wt{l} \comp u) \comp c = \wt{l} \comp \wt{c} \comp inr = [k, h \comp c] \comp inr = h \comp c
		\]
		The universal property of $\ol{c}$ implies that the factorization is unique. 
		
		We still need to verify that the coproduct $X + C = \wt{C}$ is conical.
		Let $p,q: \wt{C} \to X$ fulfil
		\[
		p \comp \wt{c} \comp inl \sqleq q \comp \wt{c} \comp inl \text{ and } p \comp u \sqleq q \comp u,
		\]
		then we verify $p \sqleq q$.
		The inequality $p \comp u \sqleq q \comp u$ precomposed by $c$ yields $p \comp \wt{c} \comp inl \sqleq q \comp \wt{c} \comp inr$.
		Since the coproduct $X + B$ is conical, we conclude $p \comp \wt{c} \sqleq q \comp \wt{c}$, and the universal property of $\wt{c}$ implies $p \sqleq q$.
		
		\item[(3a)] $\A \subseteq \K$.
		For every object $K \in \K$ the full subcategory of all $Y$ with a conical coproduct $K + Y$ in $\Rci \K$ contains $\K$ by Item~(3b).
		Thus $K + Y$ is a conical coproduct for every object $Y$.
	\end{enumerate}

	\item $\Rci \K$ is e-finitely cocomplete.
	We use Proposition~\ref{P:e}.
	We know from (1) and (3) that finite conical coproducts exist.
	Next, conical coequalizers exist: given a pair $g,h: X \to Y$ in $\Rci \K$, their conical coequalizer $c$ is the following reflexive coinserter:
	\[
	\begin{tikzcd}
	{X + X + Y} \ar[r, shift left, "{[h,g,\id]}"] \ar[r, swap, shift right, "{[g,h,\id]}"] & Y \ar[r, "c"] & C
	\end{tikzcd}
	\]
	Consequently, $\Rci \K$ has finite conical colimits, and since in $\Pos$ the concepts finite and e-finite coincide, it remains to prove that $P \tensor \blank$ exist for every finite poset $P$.
	We have the canonical reflexive coinserter in Remark~\ref{R:id}.
	For every object $X$ of $\Rci \K$ we have, due to~(3) above, the finite coproducts
	\[
	R \tensor X = \coprod_R X \text{ and } |P| \tensor X = \coprod_{|P|} X
	\]
	in $\Rci \K$ and we form the corresponding reflexive coinserter $c$ in $\Rci \K$:
	\[
	\begin{tikzcd}
	{R \tensor X} \ar[r, shift left, "\pi_1 \tensor X"] \ar[r, swap, shift right, "\pi_0 \tensor X"] & {|P| \tensor X} \ar[r, "c"] \ar[d, swap, "{[f_p]}"] & C \ar[dl, "f"] \\
	& Y &
	\end{tikzcd}
	\]
	We prove that
	\[
	C = P \tensor X.
	\]
	Indeed, to give a morphism $f : C \to Y$ in $\Rci \K$ means to give morphisms $f_p : X \to Y$ for all $p \in |P|$ such that $[f_p] \comp (\pi_0 \tensor X) \leq [f_p] \comp (\pi_1 \tensor X)$.
	Equivalently: whenever $p_0 \leq p_1$ holds in $P$, then $f_{p_0} \leq f_{p_1}$ in $\Rci \K (X,Y)$.
	This is the same as specifying a monotone function from $P$ to $\Rci \K (X,Y)$, as claimed.
	
	\item We finally prove that for every e-finitely cocomplete category $\LL$ and every functor $F: \K\to \LL$ preserving finite coproducts there exists a unique e-finitely cocontinuous extension $F: \Rci \K\to \LL$.
	This concludes the proof by Remark~\ref{R:free}.
	
	By the universal property of $\Rci \K$ we have a unique extension $\ol{F}: \Rci \K\to \LL$ preserving reflexive coinserters.
	We now prove that it preserves finite coproducts.
	It follows that $\ol{F}$ preserves tensors $P \tensor \blank$ with $P$ finite: see Item~(4).
	Moreover, $\ol{F}$ preserves reflexive coequalizers of $f,g: X \to Y$: they are the coinserters of $[f,g], [g,f]: X + X \to Y$.
	Therefore $\ol{F}$ preserves finite colimits, thus by Corollary~\ref{C:e} $\ol{F}$ is e-finitely cocontinuous.
	
	The proof that $\ol{F}$ preserves, for every object $X$ of $\Rci \K$, all coproducts $X + Y$, is by structural induction.
	Let $\A$ be the full subcategory of all $X$ with $\ol{F}(X + Y) = \ol{F}X + \ol{F}Y$ holding for every $Y \in \Rci \K$.
	
	\begin{enumerate}

		\item [(5a)] $\A$ contains $\K$.
		Indeed, assuming $X \in \K$, we prove the desired property by structural induction again.
		Let $\A'$ be the full subcategory of all $Y$ with $\ol{F}(X + Y) = \ol{F} X + \ol{F} Y$.
		Since $F$ preserves finite coproducts, $\K\subseteq \A'$: for $Y \in \K$ we have $\ol{F} (X + Y) = F(X + Y)$ using Item~(2), and $\ol{F} X + \ol{F} Y = FX + FY = F(X + Y)$.
		And $\A'$ is closed under reflexive coinserters: given a coinserter~\eqref{coins} with $A, B \in \A'$ we construct $X + C = \wt{C}$ as in Item~(3b) and use that $\ol{F}$ preserves both reflexive coinserters and the coproducts $X + A$ and $X + B$ to conclude that $\ol{F} X + \ol{F} C = \ol{F} \wt{C}$.
		Thus $\A' = \Rci \K$.
		
		\item [(5b)] $\A$ is closed under reflexive coinserters.
		This follows from $\ol{F}$ preserving reflexive coinserters and the construction of finite coproducts in Item~(3b).

	\end{enumerate}

\end{enumerate}
\end{proof}

\begin{theorem}
\label{T:free}
Let $\K$ be a poset-enriched category with finite coproducts. Then
\[
\Sind \K= \Ind (\Rci \K).
\]
\end{theorem}

\begin{proof}
\phantom{phantom}
\begin{enumerate}
	
	\item Let $\K$ be small.
	By Theorem~\ref{T:Rci} the category $\Rci \K$ has finite colimits, thus, $\Ind (\Rci \K)$ is the category of all presheaves $H$ on $\Rci \K$ preserving finite limits (Remark~\ref{R:ind}).
	That is, all presheaves such that
	\[
	H^\op: \Rci \K \to \Pos^\op
	\]
	preserves finite colimits.
	Such presheaves are, using that theorem again, precisely the extensions of functors $F: \K\to \Pos^\op$ preserving finite coproducts.
	By Corollary~\ref{C:sind} this means that $F^\op: \K^\op \to \Pos$ lies in $\Sind \K$.
	We obtain an equivalence of categories $\Ind (\Rci \K)$ and $\Sind \K$ by assigning to $H$ the unique functor $F^\op$ such that $H^\op$ extends $F$.
	
	\item Let $\K$ be large.
	Form the collection $\K_i \subseteq \K$ ($i \in I$) of all essentially small full subcategories closed under finite coproducts in $\K$.
	Order $I$ by inclusion, then $\K_i$ is closed under finite coproducts in $\K_j$ for all $i \leq j$ in $I$.
	Theorem~\ref{T:Rci} implies that $\Rci \K_i$ is closed under finite colimits in $\Rci \K_j$.
	Since $\Ind (\Rci \K)$ is formed by small presheaves on $\Rci \K$ preserving finite limits (Remark~\ref{R:ind}]) and for each such presheaf $H: (\Rci \K)^\op \to \Pos$ the codomain restrictions to $(\Rci \K_i)^\op$ preserve finite limits, too, we conclude that $\Ind (\Rci \K)$ is the colimit of the diagram of all $\Ind (\Rci \K_i)$ indexed by the ordered class $I$.
	
	Analogously, $\Sind \K_i$ is closed under finite coproducts in $\Sind \K_j$ for all $i \leq j$ in $I$.
	Since $\Sind \K$ is formed by small presheaves preserving finite products (by Remark~\ref{R:ind} again), we conclude that $\Sind \K$ is the colimit of the diagram of all $\Sind \K_i$ for the ordered class $I$.
	Thus the theorem follows from~(1) above.
	
\end{enumerate}
\end{proof}

\begin{openproblem}
\phantom{phantom}
\begin{enumerate}
	
	\item Does the above theorem hold for poset-enriched categories with pullbacks? (Compare Theorem~\ref{T:pb}.)
	
	\item Does that theorem generalize to $\V$-categories for basic subcategories $\V$ of $\Pos$?
	
\end{enumerate}
\end{openproblem}

\begin{theorem}
\label{T:preser}
A poset-enriched functor between cocomplete categories preserves sifted colimits iff it preserves
\begin{enumerate}[label=(\alph*)]
	\item filtered colimits and
	\item reflexive coinserters.
\end{enumerate}
\end{theorem}

\begin{proof}
Let $\K$ and $\LL$ be cocomplete poset-enriched categories.
Given an enriched functor $F : \K\to \LL$ preserving filtered colimits and reflexive coinserters, we prove that it preserves sifted colimits.
\begin{enumerate}
	
	\item Every small full subcategory $C: \C\hookrightarrow \K$ closed under finite coproducts has the following property.
	Denote by $(\blank)^*$ the extension of functors from $\C$ to $\Sind \C$ preserving sifted colimits, then the triangle below commutes:
	\[
	\begin{tikzcd}
	& {\Sind \C} \ar[dl, swap, "C^*"] \ar[d, "(F \comp C)^*"] \\
	{\K} \ar[r, swap, "F"] & {\LL}
	\end{tikzcd}
	\]
	To verify this, observe that both $F \comp C^*$ and $(F \comp C)^*$ preserve filtered colimits.
	Thus, to prove they coincide, it is sufficient to verify the equality
	\[
	F \comp C^* \comp I = (F \comp C)^* \comp I : \Rci \C\to \LL
	\]
	for the embedding $I: \Rci \C \hookrightarrow \Ind (\Rci \C) = \Sind \C$ (Theorem~\ref{T:free}).
	The functor $I$ preserves finite weighted colimits (\cite{kelly+schmitt}, Proposition~5.6 applied to $\Phi = $finite limits).
	Therefore both sides of the last equation preserve reflexive coinserters.
	Thus, the equation holds because both functors are extensions of $F \comp C: \C\to \LL$.
	
	\item We are ready to prove that $F$ preserves $\Colim{W}{D}$ for every small diagram $D : \D \to \K$ and every sifted weight $W : \D^\op \to \V$.
	Let $C: \C\hookrightarrow \K$ denote the full subcategory which is the closure of $D[\D]$ under finite coproducts.
	Since $\C$ is essentially small, we can apply~(1) to it.
	We have a factorization $D = C \comp D'$, and obtain a commutatitve diagram as follows:
	\[
	\begin{tikzcd}
	& {\C} \ar[d, hook, swap, "C"] \ar[r, hook, "E"] & {\Sind \C} \ar[dl, swap, "C^*"] \ar[d, "(FC)^*"] \\
	{\D} \ar[ur, "D'"] \ar[r, swap, "D"] & {\K} \ar[r, swap, "F"] & {\LL}
	\end{tikzcd}
	\]
	We thus get the following canonical isomorphisms
	\begin{align*}
	F (\Colim{W}{D}) & = F (\Colim{W}{C^* \comp E \comp D'}) \\
	& \cong F \comp C^* (\Colim{W}{E \comp D'}) && \text{$W$ sifted} \\
	& = (F \comp C)^* (\Colim{W}{E \comp D'}) && \text{by (1)} \\
	& \cong \Colim{W}{((F\comp C)^* \comp E \comp D')} && \text{$W$ sifted} \\
	& = \Colim{W}{(F \comp D)}.
	\end{align*}
	
\end{enumerate}
\end{proof}

\section{Strongly Finitary Endofunctors}

The concept of strong finitarity was introduced by Kelly and Lack~\cite{kelly+lack:strongly-finitary}, see below.
For all the categories $\V$ we consider in our paper an endofunctor of $\V$ is strongly finitary iff it preserves sifted colimits.
In the subsequent sections we prove that strongly finitary monads on $\CPO$ or $\DCPO$ bijectively correspond to varieties of continuous (or $\Delta$-continuous) algebras.

The assumption~\ref{ASS} are still assumed throughout this section.

\begin{definition}[\cite{kelly+lack:strongly-finitary}]
An endofunctor $T$ of $\V$ is \emph{strongly finitary} if it is the left Kan extension of its restriction $T \comp K$ to finite sets:
\[
T = \Lan{K}{(T \comp K)}
\]
\end{definition}

\begin{remark}
\label{R:setfcomp}
\phantom{phantom}
\begin{enumerate}

	\item Recall that $\Lan{K}{}: [\Set_\fin, \V] \to [\V,\V]$ is the left adjoint of the functor $(\blank) \comp K : [\V, \V] \to [\Set_\fin, \V]$.
	
	\item In case $\V$ is the free completion of $\Set_\fin$ under sifted colimits (via $K: \Set_\fin \to \V$) we simply write
	\[
	\V = \Sind \Set_\fin.
	\]
	Using Proposition~\ref{P:general} this is equivalent to $\V$ being the iterated closure of $K[\Set_\fin]$ under sifted colimits.
	Indeed, (1) in that proposition is our assumption, and (2) is automatic since sifted colimits commute with finite powers.
	And $\V(KX, \blank) \cong (\blank)^n$ for every set $X$ of $n$ elements.
	
\end{enumerate}
\end{remark}

\begin{proposition}
\label{P:endosf}
Suppose $\V = \Sind \Set_\fin$.
Then an endofunctor of $\V$ is strongly finitary iff it preserves sifted colimits.
\end{proposition}

\begin{proof}
According to~\cite{kelly:book}, Theorem~5.29 the condition $T = \Lan{K}{(T \comp K)}$ is equivalent to $T$ preserving $K$-absolute colimits.
This means colimits weighted by such weights $W$ that $\V(KX,\blank)$ preserves colimits weighted by $W$ for each $X \in \Set_\fin$.
If $X$ is an $n$-element set, then $\V(KX,\blank)$ is naturally isomorphic to $(\blank)^n$.
Thus
\begin{center}
$K$-absolute = sifted.
\end{center}
(Indeed, for $n > 0$ sifted colimits commute with $n$-th powers, for $n=0$ the functor $\V(KX,\blank)$ is constant with value $0$, the initial object.)

\end{proof}

\begin{examples}
\label{E:strong}
\phantom{phantom}
\begin{enumerate}

	\item For endofunctors of $\V = \Set$ finitary and strongly finitary are equivalent conditions.
	Indeed, $\Set = \Ind \Set_\fin$.
	
	\item An enriched endofunctor of $\Pos$ is strongly finitary iff it is finitary and preserves reflexive coinserters.
	This follows from Theorem~\ref{T:preser} and Proposition~\ref{P:endosf}, using Remark~\ref{R:setfcomp}~(2).
	Indeed, we have
	\[
	\Pos = \Sind \Set_\fin,
	\]
	as $\Pos$ is the closure of finite posets under filtered colimits, and finite posets form, by Remark~\ref{R:id}, the closure of $K[\Set_\fin]$ under reflexive coinserters.
	
	\item As mentioned in Example~\ref{E:Bourke}, an endofunctor of $\Cat$ is strongly finitary iff it is finitary and preserves codescent objects of strict reflexive data.
	
\end{enumerate}
\end{examples}

\begin{proposition}
\label{P:CPO}
$\CPO = \Sind \Set_\fin$, and an enriched endofunctor of $\CPO$ is strongly finitary iff it is finitary and preserves reflexive coinserters.
\end{proposition}

\begin{proof}
Recall the weight $W_0$ for reflexive coinserters (Definition~\ref{D:coins}).
We are going to prove that $\CPO$ is the closure of $\Set_\fin$ (finite discrete posets) under filtered colimits and reflexive coinserters.
Then $\CPO = \Sind \Set_\fin$ follows from Proposition~\ref{P:general} applied to
\[
\Phi = \text{filtered weights plus } W_0,
\]
using that $\CPO$ is cocomplete. Indeed, once we know that $\CPO$ is the free completion of $\Phi$-colimits, it is a free completion under sifted ones.
Consequently, an enriched endofunctor $F$ of $\CPO$ preserving filtered colimits and reflexive coinserters is the unique $\Phi$-cocontinuous extension of $F \comp K$.
The unique strongly finitary extension of $F \comp K$ to $\Sind \Set_\fin$ is also $\Phi$-cocontinuous, thus, it is the functor $F$.

Let us denote by $\C$ the closure of $\Set_\fin$ under filtered colimits and reflexive coinserters in $\CPO$.
We prove $\C = \CPO$ in several steps.
\begin{enumerate}

	\item $\C$ contains all finite posets by Remark~\ref{R:id}.
	
	\item $\C$ contains the cpo $\Nat^\top$ (Example~\ref{E:N}).
	Analogously, the copower $r \tens \Nat^\top$ ($r < \omega$) lies in $\C$: it is the colimit of the $\omega$-chain of $r \tens C_k$ ($k < \omega$), the coproduct of $r$ copies of $C_k$, with inclusions as connecting maps.
	
	\item Basic cpos. Let us call a cpo $P$ \emph{basic} if we can obtain it as a reflexive coinserter of the following form
	\[
	\begin{tikzcd}
		{r' \tens \Nat^\top} \ar[r, bend left, "f_1"] \ar[r, bend right, swap, "f_0"] & {r \tens \Nat^\top} \ar[r, "p"] & P
	\end{tikzcd}
	\]
	for $r,r' \in \Nat$.
	We also call such coinserters \emph{basic}.
	$\C$ is of course closed under basic coinserters.
	Observe that a coproduct of two basic coinserters is basic, too.
	
	\item To finish the proof we verify that every cpo $X$ is a directed colimit of basic ones.
	Denote by $D$ the directed diagram of all sub-cpos $P \subseteq X$ that are basic (and all inclusion morphisms between them).
	The fact that $D$ is directed follows from the observation above.
	The embeddings $i_P : P \hookrightarrow X$ of all basic sub-cpos form a cocone of $D$.
	We prove that every other cocone $s_P : P \to S$ of $D$ uniquely factorizes through $(i_P)$.
	
	First observe that $D$ contains $\{ x \}$ for every $x \in X$: consider the basic coinserter
	\[
	\begin{tikzcd}
		{2 \tens \Nat^\top } \ar[r, bend left, "f_1"above] \ar[r, bend right, swap, "f_0"] & {\Nat^\top} \ar[r, "p"] & {\{ x \}}
	\end{tikzcd}
	\]
	where $f_0$ and $f_1$ are $\id$ on the first copy of $\Nat^\top$, and on the second one $f_0$ is $\id$ and $f_1$ is constant with value $0$.
	Next observe that given $P \in D$ containing $\{ x \}$, it follows that $s_P(x) = s_{\{x\}}(x)$.
	(In other words, the value $s_P(x)$ is independent of $P$.)
	Indeed, the embedding $\{ x \} \hookrightarrow P$ is a connecting morphism of $D$, thus for the cocone $(s_P)$ we see that $s_{\{x\}}$ is a restriction of $s_P$.
	We can thus define a map
	\[
	s : X \to S \text{ by } s(x) = s_P(x)
	\]
	for any $P \in D$ containing $x$.
	This map is monotone.
	Indeed, given $x \sqleq y$ in $X$, then $D$ contains $\{ x, y \}$.
	To see this, consider the basic coinserter
	\[
	\begin{tikzcd}
		{2 \tens \Nat^\top} \ar[r, bend left, "f_1"above] \ar[r, bend right, swap, "f_0"] & {\Nat^\top} \ar[r, "p"] & {\{ x, y \}}
	\end{tikzcd}
	\]
	where $f_0$ and $f_1$ are $\id$ on the first copy of $\Nat^\top$ and on the second one $f_0$ is $\id$ and $f_1$ is constant with value $1$ except for $f_1(0) = 0$.
	Then $p(0) = x$ and else $p$ is constant with value $y$.
	Since $s_{\{x,y\}}$ is monotone, we get
	\[
	s(x) = s_{\{x,y\}}(x) \sqleq s_{\{x,y\}}(y) = s(y).
	\]
	Analogously, $s$ is continuous: let $x = \bigsqcup_{k \in \Nat} x_k$ be a join of a strictly increasing $\omega$-chain in $X$.
	Consider the basic coinserter
	\[
	\begin{tikzcd}
		{\Nat^\top} \ar[r, bend left, "\id"] \ar[r, bend right, swap, "\id"] & {\Nat^\top} \ar[r, "p"] & {\{ x_k \}_{k \in \Nat} \cup \{ x \}}
	\end{tikzcd}
	\]
	where $p(k) = x_k$ and $p(\top) = x$.
	
	By the definition of $s$ we have $s \comp i_P = s_P$ for very $P \in D$.
	The map $s$ is obviously unique with this property.
	
\end{enumerate}
\end{proof}

\begin{proposition}
\label{P:DCPO}
$\DCPO = \Sind \Set_\fin$, and an enriched endofunctor of $\DCPO$ is strongly finitary iff it is finitary and preserves reflexive coinserters.
\end{proposition}

\begin{proof}
Analogously to the proof of Proposition~\ref{P:CPO}, we need to prove that
\[
\C= \DCPO,
\]
where $\C$ is the closure of finite discrete posets under filtered colimits and reflexive coinserters.

\begin{enumerate}

	\item For every ordinal $\alpha$ consider the linearly ordered dcpo $\alpha + 1$ of all ordinals smaller or equal to $\alpha$.
	We prove that if $\alpha$ is infinite, then $\alpha+1$ is a colimit of a chain of dcpos $\beta + 1$ lying in $\C$.
	We proceed by transfinite induction on $\alpha$.
	\begin{description}

		\item[First step] $\alpha = \omega$.
		Since $\alpha + 1 \cong \Nat^\top$, see Example~\ref{E:N}.
		
		\item[Isolated step] Suppose $\alpha + 1$ is a colimit of some $\gamma$-chain $D$ having objects $Di = \beta_i + 1$ for $i < \gamma$.
		We observe that $\alpha + 2$ is a colimit of the following $\gamma$-chain $D'$: for each $i < \gamma$ put
		\[
		D'i = \beta_i + 2 \text{ (having top element $\beta_i + 1$).}
		\]
		The connecting maps $d_{ij} : Di \to Dj$ of $D$ are extended to the connecting maps $d_{ij}' : D'i \to D'j$ by preerving the top elements: $d_{ij}' (\beta_i + 1) = \beta_j + 1$.
		Then the colimit of $D'$ is obtained from $\alpha + 1 = \colim D$ by adding a new top element.
		That is, $\alpha + 2 = \colim D'$.
		
		\item[Limit step] Given a limit ordinal $\alpha$ with $\beta + 1 \in \C$ for all $\beta < \alpha$, then $\alpha + 1 \in \C$ because the colimit of the $\alpha$-chain of all $\beta+ 1$ (and inclusion maps as the connecting maps) is $\alpha + 1$.
		
	\end{description}
	
	\item We next prove that $\C$ contains all finite coproducts of the above linearly ordered cpos $C_\alpha = \alpha + 1$ with $\alpha$ an infinite ordinal.
	We provide the detailed proof for binary coproducts
	\[
	C_\alpha + C_\beta.
	\]
	Without loss of generality assume $\alpha \leq \beta$.
	We prove by transfinite induction on $\beta$ that $C_\alpha + C_\beta$ is a colimit of a chain of dcpos $C_\alpha + C_\delta$ for ordinals $\delta < \beta$.
	
	\begin{description}
		
		\item[Initial step] $\beta = \omega$.
		Thus $\alpha = \omega$, too, and $C_\alpha = C_\beta \cong \Nat^\top$ (Example~\ref{E:N}).
		We have $\Nat^\top + \Nat^\top = \colim_{k < \omega} (C_k + C_k)$ in $\DCPO$.
		
		\item[Isolated step] Suppose that $C_\alpha + C_\beta$ is a colimit of some $\gamma$-chain $D$ having objects $D_i = C_\alpha + C_{\beta_i}$ for $i < \gamma$.
		Form the $\gamma$-chain $D'$ with objects $C_\alpha + C_{\beta_i + 1}$ and with connecting morphisms extending those of $D$ by preserving the top element $\beta_i + 1$ of $C_{\beta_i + 1}$.
		Then $C_\alpha + C_{\beta+ 1} = \colim D'$ in $\DCPO$.
		
		\item[Limit step] Let $\beta$ be a limit ordinal such that $\C$ contains $C_\alpha + C_\delta$ for all infinite $\delta < \beta$.
		These coproducts form a $\delta$-chain (with connecting maps given by the inclusion maps) having the colimit $C_\alpha + C_\beta$.
		
	\end{description}
	
	\item Basic dcpos.
	Let $\C_0 \subseteq \C$ denote the full subcategory of $\DCPO$ which is the closure of the class $\{ C_\alpha \mid \alpha \text{ an infinite ordinal} \}$ under finite coproducts.
	A basic dcpo is a dcpo $P$ for which there exists a reflexive coequalizer $p : C_{\alpha_1} + \dots + C_{\alpha_r} \to P$ of a parallel pair in $\C_0$.
	Thus $\C$ contains all basic dcpos.
	
	\item Every dcpo $X$ is a directed colimit of basic ones.
	The proof is completely analogous to that of Item~(4) in Proposition~\ref{P:CPO}.
	In the last argument showing that $s$ is continuous we just verify that $s$ preserves the join $x = \bigsqcup_{k < \alpha} x_k$ of every $\alpha$-chain, where $\alpha$ is an infinite cardinal.
	(This follows from a coinserter of $\id, \id: C_\alpha \to C_\alpha$ as in loc.\ cit.)
	Then $s$ preserves joins of increasing chains, which proves that it preserves directed joins (\cite{adamek+rosicky}, Corollary~1.7).
	
\end{enumerate}
\end{proof}

\begin{remark}
\label{R:surj}
\phantom{phantom}
\begin{enumerate}

	\item In $\Pos$ all coinserters are surjective.
	This is not true in $\CPO$: let $|\Nat|$ be the underlying discrete $\CPO$ and $R$ be the order of $\Nat$ (a discrete $\CPO$). Let $\pi_0,\pi_1: R \to |\Nat|$ be the projections onto the underlying set.
	Their coinserter is $|\Nat| \hookrightarrow \Nat^\top$.
	
	\item On the other hand, all coinserters used in the proof of Proposition~\ref{P:CPO} are surjective.
	In step~(1) this was $\id_P$.
	In step~(3) the map $p$ is surjective because its image is closed under $\omega$-joins in $P$.
	Indeed, that image is the union of $p[\{i\} \times \Nat^\top]$ for $i = 1,\dots,r$.
	Let $(x_n)$ be a strictly increasing $\omega$-chain in that image.
	Since the above union is finite, some $p[\{i\} \times \Nat^\top]$ contains a cofinal subchain $(x_{n(k)})_{k \in \Nat}$ of the chain $(x_n)$.
	Thus we have a strictly increasing chain $(j_k)_{k \in \Nat}$ in $\Nat$ with $x_{n(k)} = p(i,j_k)$ for all $k \in \Nat$.
	We conclude that $\bigsqcup_{n < \omega} x_n$ lies in the image of $p$:
	\[
	\bigsqcup_{n < \omega} x_n = \bigsqcup_{k < \omega} x_{n(k)} = \bigsqcup_{k < \omega} p(i,j_k).
	\]
	The last join equals $p(i,\top)$ since $p$ is continuous and $\bigsqcup_{k < \omega} (i,j_k) = (i, \top)$ in $\{ i \} \times \Nat^\top$.
	
	\item Analogously, all coinserters used in the proof of Proposition~\ref{P:DCPO} are surjective.
	We thus obtain the following
	
\end{enumerate}
\end{remark}

\begin{corollary}
	\label{C:cor}
	Strongly finitary endofunctors on $\Pos$, $\CPO$ or $\DCPO$ are precisely those preserving directed colimits and reflexive, surjective coinserters.
\end{corollary}

\begin{example}
\label{E:sf}
For the base categories $\Pos$, $\CPO$ or $\DCPO$:
\begin{enumerate}

	\item The endofunctor $(\blank)^n$ is strongly finitary ($n \in \Nat$).
	This follows, in view of the above Corollary, from Proposition~\ref{P:fil} and Theorem~\ref{T:Rci}.
	
	\item A coproduct of strongly finitary endofunctors is strongly finitary.
	
\end{enumerate}
\end{example}

\section{Varieties of Continuous Algebras}

We now introduce varieties of continuous algebras: the base category is $\CPO$ (see Appendix for details on it).
In the next section the analogous results about varieties of $\Delta$-continuous algebras (base category $\DCPO$) are presented.
A variety is a class of continuous algebras presented by equations between extended terms.
These terms use, besides the usual formation of composite terms, formal joins of $\omega$-chains $t = \bigvee_{k \in \Nat} t_k$ for countable sets of terms.
We use the symbol $\bigsqcup$ for joins in a concrete poset and $\bigvee$ for formal joins.
The underlying set of a cpo $P$ is denoted by $|P|$.

In the present section we prove that every variety of continuous algebras has the form $\CPO^\Tmon$ (the Eilenberg-Moore category) for a strongly finitary monad.
The converse is proved in Section~8: every strongly finitary monad yields a variety.

Throughout this section `category' means a $\CPO$-enriched category, and `functor' means a $\CPO$-enriched functor (i.e., a locally continuous one, see Appendix). 

\begin{assumption}
	For the rest of the paper $\Sigma$ denotes a finitary signature: every symbol $\sigma \in \Sigma$ is assigned an arity (which is a natural number).
	We assume that a countably infinite set
	\[
	V = \{ x_k \mid k \in \Nat \}
	\]
	of variables is specified.

\end{assumption}

The following definition stems essentially from the work of the ADJ group in the 1970s~\cite{ADJ}.

\begin{definition}
	A \emph{continuous algebra} is an algebra acting on a cpo $A$ with all operations continuous.
	That is, for every $n$-ary symbol $\sigma \in \Sigma$, we are given a map $\sigma_A : A^n \to A$ continuous w.r.t.\ the coordinate-wise order on $A^n$.
	
	We denote by $\Sigma\text{-}\CPO$ the category of continuous algebras and continuous homomorphisms.
\end{definition}

\begin{example}[Free algebras]
	\label{E:free}
	The description of a free algebra $T_\Sigma P$ on a given cpo $P$ is analogous to that in (non-ordered) universal algebra.
	The elements of $T_\Sigma P$ are \emph{classical terms} (we stress 'classical' since below we use more general terms).
	That is, the underlying set $|T_\Sigma P|$ is the smallest set such that
	\begin{enumerate}
		
		\item every variable, i.e.\ an element of $|P|$ is a classical term, and 
		
		\item for every $n$-ary symbol $\sigma \in \Sigma$ and every $n$-tuple $(t_i)_{i < n}$ of classical terms we get a composite term $\sigma(t_i)_{i < n}$.
		
	\end{enumerate}
	Composite terms define the operations on $|T_\Sigma P|$.
	
	Let us call two classical terms \emph{similar} iff either both are variables from $P$ or both have, for some $n$-ary symbol $\sigma$, the form $\sigma(t_i)_{i < n}$ and $\sigma(t_i')_{i < n}$, resp., such that $t_i$ is similar to $t_i'$ for all $i < n$.
	The order $\sqleq$ of $T_\Sigma P$ is as follows: only pairs of similar classical terms are comparable.
	For variables, $x \sqleq x'$ holds in $T_\Sigma P$ iff this holds in $P$.
	And for similar composite terms we put
	\[
	\sigma(t_i) \sqleq \sigma(t_i') \text{ iff } t_i \sqleq t_i' \text{ for each } i < n.
	\]
\end{example}

\begin{remark}
	\label{R:sigma}
	\phantom{phantom}
	\begin{enumerate}
		
		\item It is easy to see that two terms $t$ and $t'$ are similar iff we can obtain $t'$ from $t$ by changing some variables (in $|P|$) by other variables.
		And that $t \sqleq t'$ holds iff we can obtain $t'$ by changing some variables $x$ in $t$ by variables $x'$ with $x \sqleq x'$ in $P$.
		
		\item Consequently, $T_\Sigma P$ is a coproduct of powers $P^r$, one for each similarity class of terms on $r$ variables.
		Moreover, these classes are independent of the choice of $P$.
		We obtain the free-algebra functor $T_\Sigma : \CPO \to \Sigma\text{-}\CPO$ as a coproduct
		\[
		T_\Sigma = \Coprod \Id^r
		\]
		ranging over similarity classes of terms (on $r$ variables).
		
	\end{enumerate}
\end{remark}

\begin{definition}
	A monad is strongly finitary if its underlying functor has this property.
\end{definition}

\begin{proposition}
	\label{C:free}
	The monad $\Tmon_\Sigma$ of free $\Sigma$-algebras on $\CPO$ is strongly finitary.
\end{proposition}

This follows from the previous remark and Example~\ref{E:sf}.

We are going to define varieties of continuous algebras as classes presentable by equations between terms.
The following definition extends the above concept of classical terms by allowing terms $t = \bigvee_{k \in \Nat} t_k$
for every collection $(t_k)_{k \in \Nat}$ of terms on finitely many variables.
Our definition is very similar to that in~\cite{anr} where, however, the restriction to collections with finitely many variables was not required.

\begin{definition}
	\label{D:ext}
	For the set $V$ of variables we define the set $\T_\Sigma V$ of \emph{(extended) terms} as the smallest set such that
	\begin{enumerate}
		
		\item Every variable in $V$ is a term.
		
		\item Every $n$-ary symbol $\sigma$ and every $n$-tuple $(t_i)_{i<n}$ of terms yields a composite term $\sigma(t_i)_{i<n}$.
		
		\item Every countable collection $t_k$ ($k \in \Nat$) of terms, all of which contain only finitely many variables, yields a term
		\[
		t = \bigvee_{k \in \Nat} t_k.
		\]
		
	\end{enumerate}
\end{definition}

\begin{notation}
	We denote by $\T_\Sigma V_n$ the subset of $\T_\Sigma V$ of extended terms using only variables from $V_n = \{ x_0, \dots, x_{n-1} \}$.
	By the above definition we have
	\[
	\T_\Sigma V = \bigcup_{n < \omega} \T_\Sigma V_n.
	\]
\end{notation}

\begin{example}
	\phantom{phantom}
	\begin{enumerate}
		
		\item Let $\Sigma$ consist of a unary operation $\sigma$.
		Then we have the following term
		\[
		\bigvee_{k \in \Nat} \sigma^k(x).
		\]
		To interpret this in an algebra $A$, we need not only to have a continuous operation $\sigma_A: A \to A$, but also have to know that for every interpretation $x \mapsto a$ of the variable the sequence $\sigma_A^k (a)$ in $A$ is an $\omega$-chain.
		This indicates that interpretation of terms will be a \emph{partial} map.

		\item For the above signature and our set $V = \{ x_k \mid k \in \Nat \}$ of variables the expression
		\[
		\bigvee_{k \in \Nat} \sigma(x_k)
		\]
		is not a term: it contains infinitely many variables.
		
	\end{enumerate}
\end{example}

\begin{definition}
	\label{D:inter}
	Let $A$ be a continuous algebra.
	Given an interpretation $f : V \to A$ of variables, we define the \emph{interpretation of extended terms} as the following partial function
	\[
	f^\sharp: \T_\Sigma V \rightharpoonup A:
	\]
	\begin{enumerate}
		
		\item $f^\sharp (x) = f(x)$ for each variable $x \in V$.
		
		\item $f^\sharp$ is defined in $\sigma(t_i)_{i < n}$ iff each $f^\sharp(t_i)$ is defined. Then $f^\sharp(\sigma(t_i)_{i < n}) = \sigma_A(f^\sharp(t_i))_{i < n}$.
		
		\item $f^\sharp$ is defined in $\bigvee_{k \in \Nat} t_k$ iff each $f^\sharp(t_k)$ is defined and fulfils $f^\sharp(t_k) \sqleq f^\sharp(t_{k+1})$.
		Then $f^\sharp(t) = \bigsqcup_{k \in \Nat} f^\sharp(t_k)$.
		
	\end{enumerate}
\end{definition}

\begin{definition}
	\label{D:eq}
	A \emph{variety of continuous algebras} is a full subcategory $\Vvar$ of $\Sigma\text{-}\CPO$ presented by a set of equations $t = t'$ between extended terms ($t,t' \in \T_\Sigma V$): an algebra $A$ lies in $\Vvar$ iff for each of the given equations $t = t'$ and each interpretation $f : V \to A$ both $f^\sharp(t)$ and $f^\sharp(t')$ are defined and are equal.
\end{definition}

\begin{example}
	\phantom{phantom}
	\begin{enumerate}
		
		\item Let $\Sigma$ consist of a unary symbol $\sigma$ and a nullary one $\top$.
		A continuous algebra is a cpo $A$ together with a continuous self-map $\sigma_A$ and an element $\top_A$.
		The equation
		\[
		\bigvee_{k \in \Nat} \sigma^n(x) = \top
		\]
		is satisfied iff every element $a \in A$ fulfils $a \sqleq \sigma_A(a)$ (thus, $\sigma_A^k(a)$ is an $\omega$-chain) and the join $\bigsqcup_{k \in \Nat} \sigma^k (a)$ is always $\top_A$.
		
		\item Continuous monoids are given by the signature $\Sigma = \{ \cdot, e \}$ and the usual equations (associativity of $\cdot$ and $e$ being a unit).
		Thus a continuous monoid is a monoid acting on a cpo so that for all $\omega$-chains $(a_k)$, $(b_k)$ we have
		\[
		( \bigsqcup_{k \in \Nat} a_k ) ( \bigsqcup_{k \in \Nat} b_k ) = \bigsqcup_{k \in \Nat} a_k b_k.
		\]
		Notice that the last equality simply expresses that the multiplication is continuous.
		(We do not have to specify the equation $(\bigvee x_k) (\bigvee y_k) = \bigvee (x_k y_k)$. Indeed, this is not an equation in our sense at all because it contains infinitely many variables.)
		
		\item Continuous monoids satisfying
		\[
		\bigvee_{k \in \Nat} x^k = e.
		\]
		These are monoids with $a \sqleq a^2$ and $\bigsqcup_{k \in \Nat} a^k = e$ for every element $a$.
		
	\end{enumerate}
\end{example}

\begin{remark}
	\label{R:def}
	\phantom{phantom}
	\begin{enumerate}
		
		\item Instead of equations we can also use formal \emph{inequations} $t \sqleq t'$ between terms.
		This is equivalent: given terms $t,t'$ define a term
		\[
		s = \bigvee_{k \in \Nat} s_k \text{ with } s_0 = t \text{ and } s_k = t' \text{ ($ k \geq 1 $).}
		\]
		Then the equation $t = s$ expresses precisely that $t \sqleq t'$.
		
		\item Another possibility, instead of equations, is using \emph{definability} of terms.
		Let us say that a term $t$ is definable in a continuous algebra $A$ iff $f^\sharp (t)$ is defined for all interpretations $f : V \to A$.
		This is the case iff $A$ satisfies the equation $t = t$.
		Conversely, given terms $t,t'$, an algebra $A$ satisfies $t = t'$ iff the term $s$ in item (1) is definable in $A$.
		
	\end{enumerate}
\end{remark}

\begin{lemma}[\cite{anr}, Proposition~3.5]
	\label{L:inter}
	Every continuous homomorphism $f : A \to B$ preserves definability of terms $t \in \T_\Sigma V$: given an interpretation $h : V \to A$ with $h^\sharp(t)$ defined in $A$, then $(fh)^\sharp (t)$ is defined in $B$ and is equal to $f(h^\sharp(t))$.
\end{lemma}

The proof in~\cite{anr} uses more general terms: it is not required that only finitely many variables are involved.
But it applies with no modifications to our situation.

\begin{remark}
	\label{R:lift}
	\label{R:fact}
	The factorization system of Lemma~\ref{L:factor} lifts to $\Sigma\text{-}\CPO$: let $f : A \to B$ be a continuous homomorphism.
	Form the closure $B_0$ of $f[A]$ under $\omega$-joins as in the proof of that lemma.
	It is sufficient to verify that $f[A]$ is closed under all operations $\sigma \in \Sigma$.
	Let $\sigma$ be $n$-ary.
	We prove that the set $X \subseteq B_0^n$ of all $n$-tuples that $\sigma_B$ maps to $B_0$ is all of $B_0^n$.
	Since $f[A]$ is $\omega$-dense in $B_0$ (Definition~\ref{D:dense}), we know that $(f[A])^n$ is $\omega$-dense in $B_0^n$ (Lemma~\ref{L:dense}).
	The set $X$ is closed under $\omega$-joins in $B_0^n$ because $\sigma_B$ is continuous.
	Thus, we only need to verify $(f[A])^n \subseteq X$.
	Indeed, given an $n$-tuple $b_i = f(a_i)$, then $\sigma_B(b_i) = f(\sigma(a_i)) \in f[A]$.
\end{remark}

\begin{remark}
	\label{R:HSP}
	In classical universal algebra varieties are precisely the classes closed under homomorphic images, subalgebras, and products (HSP classes).
	We have these constructions for continuous algebras, too:
	\begin{enumerate}
		
		\item A \emph{product} (P) of continuous algebras $A_i$ ($i \in I$) is their cartesian product $\prod_{i \in I} A_i$ with order and operations defined coordinate-wise.
		
		\item Let $A$ be a continuous algebra.
		A \emph{subalgebra} (S) is a subobject in $\Sigma\text{-}\CPO$ represented by an embedding $m : B \hookrightarrow A$ of a sub-cpo of $A$.
		
		\item A \emph{homomorphic image} (H) of $A$ is an algebra $B$ with a surjective morphism $e : A \to B$ in $\Sigma\text{-}\CPO$.
		
	\end{enumerate}
\end{remark}

\begin{proposition}
	\label{P:HSP}
	Every variety is an HSP class, i.e.\ closed in $\Sigma\text{-}\CPO$ under products, subalgebras, and homomorphic images.
\end{proposition}

\begin{proof}
	By Remark~\ref{R:def} we need to verify that given an extended term $t \in \T_\Sigma V$ the class of all algebras in which $t$ is definable is an HSP class.
	The proof presented in~\cite{anr} (pp.\ 339--340) works without changes in our (more restricting) setting.
\end{proof}

\begin{corollary}
	\label{P:free}
	Every variety $\Vvar$ has free algebras: the forgetful functor $U_\Vvar : \Vvar \to \CPO$ has a left adjoint $F_\Vvar : \CPO \to \Vvar$. 
\end{corollary}

\begin{proof}
	This is true for $\Vvar = \Sigma\text{-}\CPO$ by Remark~\ref{R:sigma}.
	Thus it is sufficient to prove that for every variety $\Vvar$ the embedding $\Vvar \hookrightarrow \Sigma\text{-}\CPO$ has a left adjoint, i.e., $\Vvar$ is reflective.
	We use the factorization system of Remark~\ref{R:fact}: we know that $\Vvar$ is closed under products and $\MM$-subobjects (Proposition~\ref{P:HSP}) and that $\Sigma\text{-}\CPO$ is $\EE$-cowellpowered since $\CPO$ is (Lemma~\ref{L:cow}).
	Thus $\Vvar$ is reflective with reflections in $\EE$ by~\cite{ahs}, Theorem~16.8.
\end{proof}

\begin{proposition}
	\label{P:D}
	Every variety is closed under filtered colimits in $\Sigma\text{-}\CPO$.
\end{proposition}

\begin{proof}
	To simplify the notation, we work with directed colimits in place of the filtered ones.
	This does note lose generality, see~\cite{adamek+rosicky}, Corollary~1.5.
	For every term $t$ we prove that the class of all algebras in which $t$ is definable is closed under directed colimits.
	This proves the proposition by Remark~\ref{R:def}~(2).
	Let a directed colimit in $\Sigma\text{-}\CPO$ be given:
	\[
	a_s : A_s \to A \text{ for } s \in S.
	\]
	Assuming that $t$ is definable in each $A_s$, we prove that it is definable in $A$.
	We proceed by structural induction: the statement is true (a) for variables in $V$, (b) for terms $t = \sigma(t_i)_{i < n}$ whenever it holds for each $t_i$, and (c) for terms $t = \bigvee_{k \in \Nat} t_k$ whenever it holds for each $t_k$.
	\begin{enumerate}[label=(\alph*)]
		
		\item This is obvious: variables are everywhere definable.
		
		\item Let $f : V \to A$ be an interpretation.
		We have $f^\sharp(t_i)$ defined for all $i$, thus by definition of $f^\sharp$ we also have
		\[
		f^\sharp(t) = \sigma_A(f^\sharp(t_i))_{i < n}.
		\]
		
		\item By assumption on $\T_\Sigma V$ there is a finite set $V_0$ of variables containing all variables in $t_k$ for $k \in \Nat$.
		Let $f : V_0 \to A$ be an interpretation.
		The union of images
		\[
		X = \bigcup_{s \in S} a_s[A_s]
		\]
		is $\omega$-dense in $A$ (Lemma~\ref{L:colim-d}).
		Hence in the finite power $[V_0,A]$ of the cpo $A$ the set $[V_0,X]$ is $\omega$-dense (Lemma~\ref{L:dense}).
		We prove that $f^\sharp (t)$ is defined by structural induction: we need to verify that (i) this is true for all $f : V_0 \to X$, and (ii) if $(f_n)$ is an $\omega$-chain in $[V_0,X]$ with all $f_n^\sharp(t)$ defined, then $f^\sharp(t)$ is defined for $f = \bigsqcup_{n \in \Nat} f_n$.
		\begin{enumerate}[label=(\roman*)]
			
			\item Since $V_0$ is finite and $X$ is a directed union, there exists $s \in S$ such that $f : V_0 \to X$ factorizes through $a_s$:
			\[
			f = a_s \comp g \text{ for some } g : V_0 \to A_s.
			\]
			By assumption $t$ is definable in $A_s$, thus $g^\sharp(t)$ is defined.
			Apply Lemma~\ref{L:inter} to the continuous homomorphism $a_s$ to conclude $f^\sharp(t) = a_s(g^\sharp(t))$.
			
			\item For $f = \bigsqcup_{n \in \Nat} f_n$ we prove $f^\sharp(t) = \bigsqcup_{n \in \Nat} f_n^\sharp(t)$ by structural induction over $\T_\Sigma V$.
			We define a chain of subsets $\T_\Sigma^j V \subseteq \T_\Sigma V$ for ordinals $j$ by transfinite recursion as follows:
			\begin{enumerate}[label=(\alph*)]
				
				\item $\T_\Sigma^0 V = V$.
				
				\item $\T_\Sigma^{j+1} V$ consists of all terms in $\T_\Sigma^{j} V$, all terms $\sigma(t_i)_{i < n}$ with each $t_i$ in $\T_\Sigma^j V$, and all terms $\bigvee_{k \in \Nat} t_k$ with each $t_k$ in $\T_\Sigma^j V$.
				
				\item $\T_\Sigma^j V = \bigcup_{i < j} \T_\Sigma^i V$ for limit ordinals $j$.
				
			\end{enumerate}
			It is easy to see that for the first uncountable ordinal $\omega_1$ we have
			\[
			\T_\Sigma V = \T_\Sigma^{\omega_1} V.
			\]
			We prove that for every term $t \in \T_\Sigma^j$ definable in all $A_s$ we have
			\[
			f_n^\sharp (t) \sqleq f_{n+1}^\sharp(t) \text{ for } n \in \Nat \text{ and } f^\sharp(t) = \bigsqcup_{n \in \Nat} f_n^\sharp(t).	\tag{\textasteriskcentered}
			\]
			We use transfinite induction on $j$.
			This statement is obvious if $t$ is a variable, and for limit ordinals $j$ there is nothing to prove.
			Our task is thus to prove (\textasteriskcentered) for every $t \in \T_\Sigma^{j+1} V$ provided that it holds for all terms in $\T_\Sigma^j V$.
			In case $t = \sigma(t_i)_{i < n}$ with all $t_i$ in $\T_\Sigma^j V$, we have $f_n^\sharp (t_i) \sqleq f_{n+1}^\sharp (t_i)$ for all $n \in \Nat$ and $i < n$, thus
			\[
			f_n^\sharp (t) = \sigma_A(f_n^\sharp(t_i)) \sqleq \sigma_A(f_{n+1}^\sharp (t_i)) = f_{n+1}^\sharp(t)
			\]
			because $\sigma_A$ is monotone.
			The proof of $f^\sharp(t) = \bigsqcup_{n \in \Nat} f_n^\sharp(t)$ is as follows.
			\begin{align*}
				f^\sharp(t) & = f^\sharp(\sigma(t_i)) \\
				& = \sigma_A(f^\sharp(t_i)) && \text{def.\ of $f^\sharp$} \\
				& = \sigma_A(\bigsqcup_{n \in \Nat} f_n^\sharp(t_i)) && \text{induction hyp.} \\
				& = \bigsqcup_{n \in \Nat} \sigma_A(f_n^\sharp(t_i)) && \text{$\sigma_A$ continuous} \\
				& = \bigsqcup_{n \in \Nat} f_n^\sharp (t) && \text{def.\ of $f_n^\sharp$.}
			\end{align*}
			In case $t = \bigvee_{k \in \Nat} t_k$ with all $t_k$ in $\T_\Sigma^j V$ we have
			\[
			f_n^\sharp(t) = \bigsqcup_{k \in \Nat} f_n^\sharp(t_k) \sqleq \bigsqcup_{k \in \Nat} f_{n+1}^\sharp(t_k) = f_{n+1}^\sharp(t). 
			\]
			Moreover, for every $k$ we know that $f_n^\sharp(t_k) \sqleq f_n^\sharp(t_{k+1})$ because $f_n^\sharp(t)$ is defined by induction hypothesis.
			It follows that for each $k$
			\[
			f^\sharp(t_k) = \bigsqcup_{n \in \Nat} f_n^\sharp(t_k) \sqleq \bigsqcup_{n \in \Nat} f_n^\sharp(t_{k+1}) = f^\sharp(t_{k+1}).
			\]
			This concludes the proof:
			\begin{align*}
				f^\sharp(t) & = \bigsqcup_{k \in \Nat} f^\sharp(t_k) && \text{def.\ of $f^\sharp$} \\
				& = \bigsqcup_{k \in \Nat} \bigsqcup_{n \in \Nat} f_n^\sharp(t_k) \\
				& = \bigsqcup_{n \in \Nat} \bigsqcup_{k \in \Nat} f_n^\sharp(t_k) \\
				& = \bigsqcup_{n \in \Nat} f_n^\sharp(t) && \text{def.\ of $f_n^\sharp(t)$.} 
			\end{align*}
			
		\end{enumerate}
		
	\end{enumerate}
	
\end{proof}

\begin{lemma}
	\label{L:triangle}
	\label{L:tri}
	Consider a commutative triangle
	\[
	\begin{tikzcd}
		& A \ar[rd, "h_2"] \ar[ld, swap, "h_1"] & \\
		B_1 \ar[rr, swap, "p"] & & B_2
	\end{tikzcd}
	\]
	where $h_i$ are morphisms of $\Sigma\text{-}\CPO$ and $p$ is a continuous map.
	If the image of $h_1$ is $\omega$-dense, then $p$ is also a morphism in $\Sigma\text{-}\CPO$.
\end{lemma}

\begin{proof}
	Since $h_1$ and $p \comp h_1$ are homomorphisms, for every $n$-ary symbol $\sigma$ the upper square and the outward rectangle in the following diagram
	\[
	\begin{tikzcd}
		A^n \ar[r, "\sigma_A"] \ar[d, swap, "h_1^n"] & A \ar[d, "h_1"] \\
		B_1^n \ar[r, "\sigma_{B_1}"] \ar[d, swap, "p^n"] & B_1 \ar[d, "p"] \\
		B_2^n \ar[r, swap, "\sigma_{B_2}"] & B_2
	\end{tikzcd}
	\]
	both commute.
	The image of $h_1^n$ is $\omega$-dense by Lemma~\ref{L:dense}.
	Since $h_1^n$ merges the two paths of the lower square (which are continuous maps), it follows that the lower square commutes, as desired.
\end{proof}

We now prove that the forgetful functor $U_\Sigma : \Sigma\text{-}\CPO \to \CPO$ creates filtered colimits.
That is: given a directed diagram $D: \D\to \Sigma\text{-}\CPO$ with a colimit of $U_\Sigma D$ given by $c_d : U_\Sigma Dd \to C$, there is a unique structure $A$ of a continuous algebra with $U_\Sigma A = C$ turning all the maps $c_d$ into homomorphisms.
Moreover, $A = \colim D$ w.r.t.\ the cocone $(c_d)$.

\begin{theorem}
	\label{T:create}
	\label{T:tri}
	The forgetful functor $U_\Sigma : \Sigma\text{-}\CPO \to \CPO$ creates filtered colimits.
\end{theorem}

\begin{proof}
	As in the previous proof we work with directed colimits only.
	Let $\sigma \in \Sigma$ be an $n$-ary operation symbol.
	From Proposition~\ref{P:fil} we know that $D^n$ has the colimit $c_d^n : U_\Sigma (Dd)^n \to C^n$.
	All the composites $c_d \comp \sigma_{Dd} : U_\Sigma (Dd)^n \to C$ form a cocone of $D$.
	Indeed, every morphism $\delta : d_1 \to d_2$ of $\D$ yields a homomorphism $D\delta : Dd_1 \to Dd_2$, thus the square below commutes:
	\[
	\begin{tikzcd}
		{U_\Sigma (Dd_1)^n} \ar[r, "U_\Sigma(D\delta)^n"] \ar[d, swap, "\sigma_{Dd_1}"] & {U_\Sigma (Dd_2)^n} \ar[d, "\sigma_{Dd_2}"] \\
		{U_\Sigma Dd_1} \ar[r, swap, "U_\Sigma D\delta"] \ar[rd, swap, "c_{d_1}"] & {U_\Sigma Dd_2} \ar[d, "c_{d_2}"] \\
		& C 
	\end{tikzcd}
	\] 
	Therefore there exists a unique morphism $\sigma_A: C^n \to C$ making the following squares commutative for each $d \in \ob \D$:
	\[
	\begin{tikzcd}
		{U_\Sigma Dd} \ar[r, "U_\Sigma Dd"] \ar[d, swap, "c_{d}^n"] & {U_\Sigma Dd_2} \ar[d, "c_{d}"] \\
		C^n \ar[r, swap, "\sigma_A"] & C 
	\end{tikzcd}
	\]
	In other words, a unique algebra structure $A$ on the cpo $C$ is given making each $c_d$ a homomorphism.
	
	To prove that the cocone $(c_d)$ is also a colimit in $\Sigma\text{-}\CPO$, let $B$ be an algebra, and $b_d : Dd \to B$ a cocone of $D$ in $\Sigma\text{-}\CPO$.
	Then $(U_\Sigma b_d)$ is a cocone of $U_\Sigma D$, thus there is a unique continuous map
	\[
	p : C \to U_\Sigma B \text{ with } p \comp U_\Sigma c_d = U_\Sigma b_d \text{ ($d \in \ob \D$).}
	\]
	It remains to verify that  $p : A \to B$ is a homomorphism.
	For that consider the following triangle
	\[
	\begin{tikzcd}
		& {\coprod_{d \in \ob \D} Dd} \ar[ld, swap, "{[c_d]}"] \ar[rd, "{[b_d]}"] & \\
		A \ar[rr, swap, "p"] & & B
	\end{tikzcd}
	\]
	We apply Lemma~\ref{L:triangle}: both $[c_d]$ and $[b_d]$ are continuous homomorphisms.
	Since $(c_d)$ is a colimit cocone, the image of $[c_d]$ is $\omega$-dense (Lemma~\ref{L:colim-d} (2)).
	Thus $p$ is a homomorphism.
	
\end{proof}

\begin{remark}
	\label{R:create}
	\label{R:tri}
	Analogously, $U_\Sigma$ creates reflexive coinserters: given a reflexive pair $f_0,f_1 : B \to B'$ in $\Sigma\text{-}\CPO$ with a coinserter $c : U_\Sigma B' \to C$ in $\CPO$, there exists a unique $\Sigma$-algebra $A$ on $C$ making $c: B' \to C$ a homomorphism.
	Moreover, this is the coinserter of $f_0,f_1 $ in $\Sigma\text{-}\CPO$.
	The proof is completely analogous, using that the endofunctor $(\blank)^n$ of $\CPO$ preserves reflexive coinserters (Theorem~\ref{T:Rci}).
\end{remark}

\begin{notation}
	For every variety $\Vvar$ of continuous algebras we denote by $\Tmon_\Vvar$ the monad of the adjunction $F_\Vvar \adj U_\Vvar$ (Corollary~\ref{P:free}).
\end{notation}

\begin{examples}
	\phantom{phantom}
	\begin{enumerate}
		
		\item For $\Vvar = \Sigma\text{-}\CPO$ this is the strongly finitary monad $\Tmon_\Sigma$ assigning to a cpo $X$ the cpo $T_\Sigma X$ of classical terms over $|X|$ (Example~\ref{E:free}).
		
		\item For the variety $\Vvar$ of continuous monoids we have
		\[
		T_\Vvar X = \coprod_{n \in \Nat} X^n,
		\]
		the monoid of words, with coordinate-wise multiplication and coordinate-wise order on words of the same lenghth.
		This endofunctor is also strongly finitary.
		
	\end{enumerate}
\end{examples}

\begin{theorem}
	\label{T:necessity}
	The monad $\Tmon_\Vvar$ of free algebras is strongly finitary for every variety $\Vvar$ of continuous algebras.
\end{theorem}

\begin{proof}
	We prove that the functor $T_\Vvar = U_\Vvar F_\Vvar$ preserves filtered colimits and reflexive, surjective coinserters (see Corollary~\ref{C:cor}).
	
	Due to the adjunction $F_\Vvar \adj U_\Vvar$ the functor $F_\Vvar$ preserves colimits, thus, we just need to show that $U_\Vvar$ preserves directed colimits and reflexive, surjective coinserters (Remark~\ref{R:surj}).
	We use the commutative triangle, where $E$ denotes the embedding:
	\[
	\begin{tikzcd}
		\Vvar \ar[r, "E"] \ar[rd, swap, "U_\Vvar"] & \Sigma\text{-}\CPO \ar[d, "U_\Sigma"] \\
		& \CPO
	\end{tikzcd}
	\]
	We know that $U_\Sigma$ preserves filtered colimits (Theorem~\ref{T:create}) and reflexive coinserters (Remark~\ref{R:create}).
	Since $\Vvar$ is closed in $\Sigma\text{-}\CPO$ under filtered colimits (Proposition~\ref{P:D}), $U_\Vvar$ preserves filtered colimits.
	Since $\Vvar$ is closed under homomorphic images (Proposition~\ref{P:HSP}), it is also closed under reflexive surjective coinserters in $\Sigma\text{-}\CPO$.
	Hence, $U_\Vvar$ preserves those coinserters.

\end{proof}

\section{Varieties of $\Delta$-Continuous Algebras}

Analogously to the preceding section, we introduce varieties of $\Delta$-continuous algebras, and prove that they have the form $\DCPO^\Tmon$ for strongly finitary monads $\Tmon$.
We again use the symbol $\bigvee$ for formal joins of terms.
But here they are also formed for ordinals larger than $\omega$.

Throughout this section 'category' means one enriched over $\DCPO$.
'Functor' means enriched functor, i.e.. a locally $\Delta$-continuous one (see Appendix).
We again assume a finitary signature $\Sigma$ and a countable set $V$ of variables are chosen.

\begin{definition}
A \emph{$\Delta$-continuous algebra} is an algebra acting on a dcpo $A$ with all operations $\Delta$-continuous.

We denote by $\Sigma\text{-}\DCPO$ the category of $\Delta$-continuous algebras and $\Delta$-continuous homomorphisms.
\end{definition}

\begin{example}[Free algebras]
	The description of free algebras in Example~\ref{E:free} applies to $\Sigma\text{-}\DCPO$ without changes.
	We again obtain $T_\Sigma P$ as a coproduct of powers $P^r$ (cf.\ Remark~\ref{R:sigma}) which proves that the underlying poset is a dcpo for each dcpo $P$.
	Therefore the free-algebra monad $\Tmon_\Sigma$ on $\DCPO$ is strongly finitary since $T_\Sigma$ is a coproduct of finite-power functors $(\blank)^n$.
\end{example}

The main modification we need when switching from $\CPO$ to $\DCPO$ is the definition of (extended) terms.
We recall here an important reduction of directed joins to joins of \emph{chains} in posets $P$ (which are monotone maps from ordinal numbers $\alpha = \{ i \in \Ord \mid i < \alpha \}$ to $P$):

\begin{theorem}[\cite{adamek+rosicky}, Corollary~1.7]
	A poset is a dcpo iff it has joins of chains.
	A function between dcpos is $\Delta$-continuous iff it preserves joins of chains.
\end{theorem}

\begin{notation}
For a signature $\Sigma$ of $\alpha$ operations we put
\[
\| \Sigma \| = 2^{\alpha + \aleph_0}.
\]
\end{notation}

\begin{proposition}[\cite{N}, Proposition~1]
\label{P:fingen}
Every $\Delta$-continuous algebra $A$ which is finitely generated has power at most $\| \Sigma \|$.
(Finite generation means that for some finite subset $X$ the only $\Delta$-continuous subalgebra containing $X$ is all of $A$.)
\end{proposition}

\begin{definition}
	The set $\T_\Sigma^\Delta V$ of \emph{($\Delta$-extended) terms} is the smallest set such that
	\begin{enumerate}
		
		\item Every variable in $V$ is a term.
		
		\item Every $n$-ary symbol $\sigma$ and every $n$-tuple $(t_i)_{i<n}$ of terms yields a composite term $\sigma(t_i)_{i<n}$.
		
		\item Given an ordinal $\alpha \leq \| \Sigma \|$, every collection $t_k$ ($k < \alpha$) of terms, all of which contain only finitely many variables, yields a term
		\[
		t = \bigvee_{k < \alpha} t_k.
		\]
		
	\end{enumerate}
\end{definition}

The definition of interpretation of terms in an algebra $A$ is completely analogous to Definition~\ref{D:inter}: in item (3) $f^\sharp$ is defined in $t = \bigvee_{k < \alpha} t_k$ iff each $f^\sharp(t_k)$ is defined and for all $k \leq l < \alpha$ we have $f^\sharp(k) \leq f^\sharp(l)$.
Then $f^\sharp(t) =\bigsqcup_{k < \alpha} f^\sharp(t_k)$.

\begin{remark}
\label{R:set}
There exist at most $\| \Sigma \|$ terms.
Indeed, since each term contains only finitely many variables, it is easy to see that $\T_\Sigma^\Delta V = \bigcup_{i < \| \Sigma \|} W_i$ where $W_0 = V$ and $W_{i+1}$ is the closure of $W_i$ under composite terms and formal joins $\bigvee_{k < \alpha} t_k$ for $\alpha \leq \| \Sigma \|$.
Moreover, an easy induction on $i$ shows $\card W_i \leq \| \Sigma \|$ for each $i$.
Thus $\card \T_\Sigma^\Delta V \leq \| \Sigma \|^2 = \| \Sigma \|$.
\end{remark}

\begin{definition}
A \emph{variety} of $\Delta$-continuous algebras is a full subcategory $\Vvar$ of $\Sigma\text{-}\DCPO$ presented by a set of equations $t = t'$ between terms: an algebra $A$ lies in $\Vvar$ iff for each of those equations and each interpretation $f : V \to A$ both $f^\sharp(t)$ and $f^\sharp(t')$ are defined and are equal.
\end{definition}

\begin{remark}
The concepts of product, homomorphic image and subalgebra are defined for $\Delta$-continuous algebras precisely as in Remark~\ref{R:HSP}.
\end{remark}

\begin{proposition}
	Every variety of $\Delta$-continuous algebras is closed under homomorphic images, subalgebras, products, and directed colimits in $\Sigma\text{-}\DCPO$.
\end{proposition}

\begin{proof}
	For the first three closure properties we can, analogously to Proposition~\ref{P:HSP}, use the proof in~\cite{anr}.
	The proof for directed colimits is a modification of that of Proposition~\ref{P:D}.
	Whereas in that proof we used structural induction on $t = \bigvee_{k \in \Nat} t_k$, we now need to use it on $t = \bigvee_{k < \alpha} t_k$ for an arbitrary ordinal $\alpha \leq \| \Sigma \|$.
	Since all $t_k$ lie in $\T_\Sigma^\Delta V_0$ for a finite set $V_0 \subseteq V$, the proof proceeds completely analogously to that of Proposition~\ref{P:D}.
	The transfinite induction for subsets $(\T_\Sigma^\Delta)^j \subseteq \T_\Sigma^\Delta V$ is analogous to that in item (ii). In our case it does not end in $\omega_1$ steps, but in $\| \Sigma \|^+$ steps.

\end{proof}

\begin{theorem}
	The forgetful functor of $\Sigma\text{-}\DCPO$ creates filtered colimits and reflexive coinserters.
\end{theorem}

This is proved precisely as Theorem~\ref{T:create} and Remark~\ref{R:create}.

\begin{theorem}
	Every variety $\Vvar$ of $\Delta$-continuous algebras has free algebras and is concretely isomorphic to $\DCPO^{\Tmon_\Vvar}$ for the corresponding monad $\Tmon_\Vvar$ on $\DCPO$.
\end{theorem}

The proof is analogous to that of Corollary~\ref{P:free} and Proposition~\ref{P:strict}, using the factorization system of Remark~\ref{R:ADCPO} (1).

\begin{theorem}
	\label{T:sf}
	For every variety $\Vvar$ of $\Delta$-continuous algebras the monad $\Tmon_\Vvar$ is strongly finitary.
\end{theorem}

\begin{proof}
We first observe that Lemma~\ref{L:tri} works in $\Sigma\text{-}\DCPO$ equally well as in $\Sigma\text{-}\CPO$.
As in Theorem~\ref{T:tri} and Remark~\ref{R:tri} we deduce that the forgetful functor $U_\Sigma : \Sigma\text{-}\DCPO \to \DCPO$ creates filtered colimits and reflexive coinserters.
We derive that $U_\Vvar : \Vvar \to \DCPO$ preserves filtered colimits and reflexive coinserters, and apply Proposition~\ref{P:DCPO} to $T_\Vvar = U_\Vvar \comp F_\Vvar$.
\end{proof}

\section{From Strongly Finitary Monads to Varieties}

Here the main result is proved for $\CPO$ and $\DCPO$: varieties of continuous algebras correspond bijectively to strongly finitary monads.
We begin with $\CPO$.
We again assume that a countable set $V$ of variables $\{ x_k \mid k \in \Nat \}$ is given.
We denote $V_n = \{ x_k \mid k \leq n \}$ and consider it as a discrete cpo.
The underlying set of a cpo $P$ is denoted by $|P|$.
Recall that joins in a concrete poset are denoted by $\bigsqcup$, whereas formal joins (defining extended terms) by $\bigvee$.

Given a monad $\Tmon = (T,\mu,\eta)$ we denote for $f : X \to TY$ by $f^*: TX \to TY$ the corresponding homomorphism of the free algebras for $\Tmon$:
\[
f^* = \mu_Y \comp Tf.
\]
For every morphism $g : Y \to Z$ we then have a commutative triangle
\[
\begin{tikzcd}
	X \ar[r, "f"] \ar[rd, swap, "(g^* \comp f)^*"] & TY \ar[d, "g^*"] \\
	& TZ
\end{tikzcd}
\]

\begin{notation}
	Every $n$-ary operation symbol $\sigma$ is identified with the term $\sigma(x_i)_{i < n}$.
\end{notation}

\begin{definition}
	\label{D:ass}
	Let $\Tmon = (T,\mu,\eta)$ be a strongly finitary monad on $\CPO$.
	The \emph{associated signature} has as $n$-ary symbols the elements of $TV_n$:
	\[
	\Sigma_n = |TV_n| \text{ ($n \in \Nat$).}
	\]
	The \emph{associated variety} $\Vvar_\Tmon$ is presented by the following equations, where $n$ and $m$ range over $\Nat$:
	\begin{enumerate}
		
		\item $\sigma = \bigvee_{k \in \Nat} \sigma_k$ for every $\omega$-chain $(\sigma_k)_{k < \omega}$ in $TV_n$ with $\sigma = \bigsqcup_{k \in \Nat} \sigma_k$.
		
		\item $u^*(\sigma) = \sigma(u(x_i))_{i<n}$ for every $\sigma \in |TV_n|$ and all maps $u : V_n \to |TV_m|$.
		
		\item $\eta_{V_n}(x_i) = x_i$ for all $i < n$.
		
	\end{enumerate}
\end{definition}

\begin{remark}
Every algebra
\[
\alpha: TA \to A
\]
in $\CPO^\Tmon$ defines a $\Sigma$-algebra as follows: given $\sigma \in \Sigma_n$ and an $n$-tuple $(a_i)$ in $A$, represented by a map $a : V_n \to A$ (taking $x_i$ to $a_i$), we put
\[
\sigma_A(a_i)_{i<n} = \alpha \comp Ta (\sigma).
\]
Every homomorphism $h : (A, \alpha) \to (B, \beta)$ in $\CPO^\Tmon$ defines a homomorphism between the associated $\Sigma$-algebras.
Indeed, given $\sigma \in TV_n$ and $a: V_n \to A$, the equality $h(\sigma_A(a_i)) = \sigma_B(h(a_i))$ follows from $h \comp \alpha = \beta \comp Th$:
\begin{align*}
	h(\sigma_A(a_i)) & = h \comp \alpha \comp Ta (\sigma) \\
	& = \beta \comp T(h \comp a) (\sigma) \\
	& = \sigma_B(h(a_i)).
\end{align*}
\end{remark}

\begin{theorem}
	\label{T:suff}
	Every strongly finitary monad on $\CPO$ is the free-algebra monad of its associated variety.
\end{theorem}

\begin{proof}
	Let $\Tmon = (T,\mu,\eta)$ be a strongly finitary monad, and let $\Tmon_\Vvar$ be the free-algebra monad of its associated variety $\Vvar$.
	Then $\Tmon_\Vvar$ is also strongly finitary by Theorem~\ref{T:necessity}.
	We prove below that for every finite discrete cpo $P$ the $\Sigma$-algebra associated with $(TP,\mu_P)$ (the free algebra on $P$ for $\Tmon$) is free on $\eta_P: P \to TP$ in $\Vvar$.
	From this it follows that the same statement holds for \emph{all} cpos $P$.
	Indeed, we have seen in the proof of Proposition~\ref{P:CPO} that all cpos are obtained from $\Set_\fin$ by (iterated) directed colimits and reflexive coinserters.
	Since both $T$ and $T_\Vvar$ preserve directed colimits and reflexive coinserters, the free algebras for $\Tmon$ and $\Tmon_\Vvar$ coincide for all cpos $P$.
	Moreover, the forgetful functors of both $\Vvar$ and $\CPO^\Tmon$ are strictly monadic (Proposition~\ref{P:strict}).
	We conclude that $\Vvar$ and $\CPO^\Tmon$ are concretely isomorphic, thus $\Tmon$ is the free-algebra monad of $\Vvar$.
	
	Consider a finite discrete cpo $P$.
	Without loss of generality, $P = V_n$ for some $n \in \Nat$.
	Given an algebra $A$ in $\Vvar_\Tmon$ and a map $f : V_n \to A$, we prove that there exists a unique continuous $\Sigma$-homomorphism $\ol{f} : TV_n \to A$ with $f = \ol{f} \comp \eta_{V_n}$.
	\begin{description}
		\item[Existence] Define $\ol{f}(\sigma) = \sigma_A(f(x_i))_{i<n}$ for every $\sigma \in TV_n$.
		The equality $f = \ol{f} \comp \eta_{V_n}$ follows since $A$ satisfies the equations~(3) in~Definition~\ref{D:ass}, thus the operation of $A$ corresponding to $\eta_{V_n}(x_i)$ is the $i$-th projection.
		The map $\ol{f}$ is continuous: given $\sigma = \bigsqcup_{k \in \Nat} \sigma_k$ in $TV_n$, the algebra $A$ satisfies $\sigma = \bigvee_{k \in \Nat} \sigma_k$.
		Therefore given an $n$-tuple $f : V_n \to A$ we have
		\[
		\ol{f}(\sigma) = \sigma_A(f(x_i)) = \bigsqcup_{k \in \Nat}(\sigma_k)_A(f(x_i)) = \bigsqcup_{k \in \Nat} \ol{f}(\sigma_k).
		\]
		To prove that $\ol{f}$ is a $\Sigma$-homomorphism, take an $m$-ary operation symbol $\tau \in TV_m$.
		We prove $\ol{f} \comp \tau_{V_m} = \tau_A \comp \ol{f}^m$.
		This means that every $k : V_m \to TV_n$ fulfils
		\[
		\ol{f} \comp \tau_{V_m} (k(x_j))_{j < m} = \tau_A \comp \ol{f}^m (k(x_j))_{j < m}.
		\]
		The definition of $\ol{f}$ yields that the right-hand side is $\tau_A(k(x_j)_A(f(x_i)))$.
		Due to equation (2) in Definition~\ref{D:ass} with $\tau$ in place of $\sigma$ this is $k^*(\tau)_A(f(x_i))$.
		The left-hand side yields the same result since
		\[
		\ol{f}^m(k(x_j)) = (k(x_j))_A (f(x_i)).
		\]
		\item[Uniqueness] Let $\ol{f}$ be a continuous $\Sigma$-homomorphism with $f = \ol{f} \comp \eta_{V_n}$.
		In $TV_n$ the operation $\sigma$ asigns to $\eta_{V_n}(x_i)$ the value $\sigma$.
		(Indeed, for every $a : n \to |TV_n|$ we have $\sigma_{TV_n}(a_i) = a^*(\sigma) = \mu_{V_n} \comp Ta(\sigma)$.
		Thus $\sigma_{TV_n}(\eta_{V_n}(x_i)) = \mu_{V_n} \comp T \eta_{V_n} (\sigma) = \sigma$.)
		Since $\ol{f}$ is a homomorphism, we conclude
		\[
		f(\sigma) = \sigma_A(\ol{f} \comp \eta_{V_n} (x_i)) = \sigma_A(f(x_i))
		\]
		which is the above formula.
	\end{description}
\end{proof}

Recall that a \emph{concrete category} over $\CPO$ is a category $\A$ endowed with a faithful functor $U: \A \to \CPO$.
Given another concrete category $(\A',U')$, a \emph{concrete functor} is a functor $H : \A \to \A'$ with $U = U'H$.
	
	\begin{example}
		For every variety $\Vvar$ (considered as a concrete category in the obvious sense) the comparison functor
		\[
		K : \Vvar \to \CPO^{\Tmon_\Vvar}
		\]
		is concrete.
		Recall that $K$ assigns to an algebra $A$ the following algebra for $\Tmon_\Vvar$ on $UA$:
		\[
		U \alpha_A : (UF)UA \to UA
		\]
		where $\alpha_A : FUA \to A$ is the unique homomorphism extending $\id_A$.
		And to a homomorphism $f : A \to B$ it assigns $Uf : (UA, U\alpha_A) \to (UB, U \alpha_B)$.
	\end{example}
	
	\begin{proposition}
		\label{P:strict}
		Every variety $\Vvar$ of continuous algebras is strictly monadic: the comparison functor $K_\Vvar : \Vvar \to \CPO^{\Tmon_\Vvar}$ is an isomorphism.
	\end{proposition}
	
	\begin{proof}
		\phantom{phantom}
		\begin{enumerate}
			
			\item For $\Vvar = \Sigma\text{-}\CPO$ (no equations) the proof is completely analogous to the classical (non-ordered) algebras, see Theorem~VI.8.1 in~\cite{maclane:cwm}.
			
			\item For a general variety we use Beck's theorem (in strict form), see~\cite{maclane:cwm}, Theorem~VI.7.1: we just need to prove that the forgetful functor $U_\Vvar : \Vvar \to \CPO$ creates coequalizers of $U_\Vvar$-split parallel pairs.
			Since $U_\Vvar$ is a right adjoint, it then follows that $K$ is an isomorphism.
			Thus our task is, given a parallel pair $d,d' : X \to Y$ in $\Vvar$ and morphisms $e$, $t$ and $s$ in $\CPO$ as follows
			\[
			\begin{tikzcd}
				{U_\Vvar X}
				\arrow[r, shift left, "Ud'"]
				\arrow[r, shift right, swap, "Ud"]
				&
				{U_\Vvar Y}
				\arrow[r, shift right=1ex, swap, "e"]
				\arrow[shiftarr={yshift=5ex}, l, swap, "t"] 
				&
				C
				\arrow[l, shift right=1ex, swap, "s"] 
			\end{tikzcd}
			\]
			satisfying the equations
			\begin{align*}
				e \comp Ud & = e \comp Ud' \\
				e \comp s & = \id_C \\
				Ud \comp t & = \id \\
				Ud' \comp t & = s \comp e
			\end{align*}
			that there exists a unique algebra $Z$ of $\Vvar$ with $C = U_\Vvar Z$ making $e$ a homomorphism $e : Y \to Z$.
			Moreover, $e$ is the coequalizer of $d$ and $d'$ in $\Vvar$.
			
			By item (1) the above condition holds for $U_\Sigma$: there is a unique continuous algebra $Z$ with $C = U_\Sigma Z$ making $e$ a homomorphism which is the coequalizer of $d$ and $d'$ in $\Sigma\text{-}\CPO$.
			Since $e \comp s = \id_C$, the homomorphism $e$ is surjective, thus $Y \in \Vvar$ implies $Z \in \Vvar$ (Proposition~\ref{P:HSP}).
			It follows that $e$ is a coequalizer of $d,d'$ in $\Vvar$, too.
			
		\end{enumerate}
	\end{proof}

\begin{remark}
\label{R:mon}
Every monad morphism $\alpha: \Tmon \to \Smon$ for monads on $\CPO$ induces a concrete functor from $\CPO^\Smon$ to $\CPO^\Tmon$: it assigns to $a : SA \to A$ the algebra $a \comp \alpha_A: TA \to A$.
Moreover, this defines a bijection between monad morphisms $\Tmon \to \Smon$ and concrete functors $\CPO^\Smon \to \CPO^\Tmon$ (\cite{barr+wells:toposes}, Theorem~3.6.3).
\end{remark}

From Proposition~\ref{P:strict} and Theorems~\ref{T:necessity} and~\ref{T:suff} we conclude the promised bijection between strongly finitary monads and varieties.
In fact, a stronger statement holds:

\begin{corollary}
	\label{C:main}
	The following categories are dually equivalent:
	\begin{enumerate}
		
		\item Strongly finitary monads on $\CPO$ and monad morphisms.
		
		\item Varieties of continuous algebras and concrete functors.
		
	\end{enumerate}
\end{corollary}

Let $\Mon_\sf$ denote the category of strongly finitary monads and $\Var$ that of varieties.
We define a functor
\[
R: \Var^\op \to \Mon_\sf
\]
on objects by $R(\Vvar) = \Tmon_\Vvar$ (see Theorem~\ref{T:necessity}).
Given a concrete functor $F: \Vvar \to \Wvar$ between varieties, the composite $K_\Wvar \comp F \comp K_\Vvar^{-1} : \CPO^{\Tmon_\Vvar} \to \CPO^{\Tmon_\Wvar}$ is also concrete.
Thus there is a unique monad morphism $\alpha: \Tmon_\Wvar \to \Tmon_\Vvar$ inducing it.
We define $R(F) = \alpha$.
It is easy to see that $R$ is a full and faithful functor.
It is an equivalence functor since every object of $\Mon_\sf$ is isomorphic to some $R(\Vvar)$ by Theorem~\ref{T:suff}.

We now turn to $\Delta$-continuous algebras.
Recall $\| \Sigma \|$ and Proposition~\ref{P:fingen}.

\begin{definition}
	Let $\Tmon = (T,\mu,\eta)$ be a strongly finitary monad on $\DCPO$.
	For the signature $\Sigma$ with $\Sigma_n = |TV_n|$ for all $n \in \Nat$ we define the \emph{associated variety} by the following equations, where $n$ and $m$ range over $\Nat$:
	\begin{enumerate}
		
		\item $\sigma = \bigvee_{k < \alpha} \sigma_k$ for every ordinal $\alpha \leq \| \Sigma \|$ and every $\alpha$-chain $(\sigma_k)_{k < \alpha}$ in $TV_n$ with $\sigma = \bigsqcup_{k < \alpha} \sigma_k$.
		
		\item $u^*(\sigma) = \sigma(u(x_i))$ for every $\sigma \in |TV_n|$ and all maps $u : V_n \to |TV_m|$.
		
		\item $\eta_{V_n}(x_i) = x_i$ for all $i < n$.
		
	\end{enumerate}
\end{definition}

\begin{theorem}
Every strongly finitary monad on $\DCPO$ is the free-algebra monad of its associated variety.
\end{theorem}

\begin{proof}
This is analogous to the proof of Theorem~\ref{T:suff}.
In the 'Existence' part, the proof that $\ol{f}$ is $\Delta$-continuous uses the fact that this is equivalent to preserving joins of chains (\cite{adamek+rosicky}, Corollary~1.7).
Let $\alpha$ be an ordinal for which a strictly increasing chain $\sigma_i$ ($i < \alpha$) exists in $TV_n$.
Then by Remark~\ref{R:set} we have
\[
\alpha \leq \card TV_n \leq \| \Sigma \|.
\]
Thus the given algebra $A$ satisfies $\sigma = \bigvee_{k<\alpha} \sigma_k$.
This implies $\ol{f}(\sigma) = \bigcup_{k < \alpha} \ol{f}(\sigma_k)$.
Since $\ol{f}$ preserves joins of increasing chains, it preserves all chain joins.
This is the only modification of the proof of Theorem~\ref{T:suff} that is needed.
\end{proof}

Concrete categories over $\DCPO$ are defined analogously to $\CPO$, and the proof of the following proposition is also analogous to that of Proposition~\ref{P:strict}:

\begin{proposition}
Every variety $\Vvar$ of $\Delta$-continuous algebras is strictly monadic.
\end{proposition}

From the above we get the following

\begin{corollary}
	The following categories are dually equivalent:
	\begin{enumerate}[label=(\roman*)]
		
		\item Strongly finitary monads on $\DCPO$ and monad morphisms.
		
		\item Varieties of $\Delta$-continuous algebras and concrete functors.
		
	\end{enumerate}
\end{corollary}

\appendix

\section{The categories $\CPO$ and $\DCPO$}

The aim of this appendix is to collect properties of the categories of complete partial orders needed in our paper.

A subset $X$ of a poset is \emph{directed} if it is nonempty, and every pair of elements has an upper bound in $X$.
A \emph{dcpo} (directed-complete partially ordered set) is a poset with directed joins.
The category
\[
\DCPO
\]
of dcpos has as morphisms the \emph{$\Delta$-continuous maps}: monotone maps preserving directed joins.
By a \emph{cpo} we mean a poset with joins of $\omega$-chains.
The category
\[
\CPO
\]
of cpos has as morphisms \emph{continuous maps}: monotone maps preserving joins of $\omega$-chains.

\begin{observation}
\label{O:app}
Both categories $\CPO$ and $\DCPO$ are cartesian closed.
Indeed, if $A$ and $B$ are cpos, then the poset $[A,B]$ of all continuous maps ordered pointwise has pointwise joins of $\omega$-chains.
Thus $[A,B]$ is a cpo.
The functor $[A,\blank]$ is right adjoint to $A \times \blank$.

Analogously for $\DCPO$.
\end{observation}

\begin{remark}
\phantom{phantom}
\begin{enumerate}

	\item A $\CPO$-enriched category has cpo structure on hom-sets making composition continuous.
	A $\CPO$-enriched functor $F: \K \to \LL$ is an ordinary functor which is \emph{locally continuous}: given $f = \bigsqcup_{n \in \Nat} f_n$ in $\K(X,Y)$, then $Ff = \bigsqcup_{n \in \Nat} Ff_n$ in $\LL(FX,FY)$.
	Enriched natural transformations are just the ordinary ones.
	
	\item Analogously for $\DCPO$: enriched functors $F: \K\to \LL$ are locally $\Delta$-continuous: the map $\K(X,Y) \to \LL(FX,FY)$ preserves directed joins.
	
\end{enumerate}
\end{remark}

\begin{proposition}[\cite{DP}]
The category $\DCPO$ is reflective in $\Pos$: the reflection $Id(P)$ of a poset $P$ is a dcpo of all \emph{ideals}, i.e.\ directed down-sets, ordered by inclusion.
The reflection map assigns to $x \in P$ the (principal) ideal $\downarrow x$.
\end{proposition}

\begin{remark}
Analogously, $\CPO$ is reflective in $\Pos$: the reflection of $P$ is the cpo of all \emph{$\omega$-ideals} (ordered again by inclusion).
These are ideals $M \subseteq P$ for which there exists an $\omega$-chain of $P$ in $M$ such that every subideal of $M$ containing that chain is all of $M$.
\end{remark}

\begin{corollary}
\label{C:co}
Both $\CPO$ and $\DCPO$ are closed in $\Pos$ under weighted limits, and they both have weighted colimits. 
\end{corollary}

\begin{definition}
\label{D:dense}
A \emph{sub-cpo} of a cpo $P$ is a subposet closed under joins of $\omega$-chains.
A subset $X \subseteq P$ is \emph{$\omega$-dense} if no proper sub-cpo od $P$ contains it.
\end{definition}

Analogously, a sub-dcpo and $\Delta$-density are defined for a dcpo.

\begin{lemma}
\label{L:factor}
The category $\CPO$ has a factorization system $(\EE,\MM)$ where $\EE$ consists of morphisms with an $\omega$-dense image, and $\MM$ of monomorphisms representing sub-cpos.
\end{lemma}

\begin{proof}
For every morphism $f : A \to B$ let $m : B_0 \hookrightarrow B$ be the smallest sub-cpo of $B$ containing $f[A]$.
The codomain restriction $e : A \to B_0$ of $f$ is clearly $\omega$-dense.
To verify the diagonal fill-in, let a commutative square
\[
\begin{tikzcd}
A \ar[r, "e"] \ar[d, swap, "u"] & A' \ar[d, "u'"] \ar[ld, dotted, swap, "d"] \\
B \ar[r, swap, "m"] & B'
\end{tikzcd}
\]
be given with $e[A]$ $\omega$-dense in $A'$ and $m : B \to B'$ representing a sub-cpo.
Without loss of generality assume $B \subseteq B'$ and $m$ is the inclusion map.
Then $u'[A'] \subseteq B$.
Indeed, $A'$ is the iterated closure of $e[A]$ under $\omega$-joins, so we just need to observe that (i) for $x \in e[A]$ we have $u'(x) \in B$ and (ii) given an $\omega$-chain $(x_n)$ in $A'$ with $u'(x_n) \in B$ for all $n$, then $u'(\bigsqcup_{n < \omega} x_n) \in B$.
Now (i) follows from the square above: given $x = e(y)$, we have $u'(x) = m \comp u (y) \in B$.
And (ii) follows from the continuity of $u'$: we have $u'(\bigsqcup_{n < \omega} x_n) = \bigsqcup_{n < \omega} u'(x_n) \in B$.
The desired diagonal $d : A' \to B$ is the codomain restriction of $u'$: since $u'$ is continuous, so is $d$.
\end{proof}

Recall the concept of a coinserter (Definition~\ref{D:coins}).
 
\begin{lemma}
\label{L:colim-d}
\phantom{phantom}
\begin{enumerate}

	\item Every coinserter in $\CPO$ has an $\omega$-dense image.
	
	\item Every directed colimit $c_i : A_i \to C$ ($i \in I$) in $\CPO$ has the union of images $\bigcup_{i \in I} c_i[A_i]$ $\omega$-dense.
	
\end{enumerate}
\end{lemma}

\begin{proof}
\phantom{phantom}

\begin{enumerate}

	\item Let $c : B \to C$ be a coinserter of $f_0,f_1 : A \to B$.
	Given a sub-cpo $m : C_0 \to C$ containing $c[B]$, we prove that $m$ is invertible.
	Let $e : B \to C_0$ be the codomain restriction of $c$.
	It is continuous because $f = m \comp e$ is.
	And it fulfils $e \comp f_0 \sqleq  e \comp f_1$ because $m \comp (e \comp f_0) \sqleq m \comp (e \comp f_1)$, and $m$ is an embedding.
	Thus $e$ factorizes through $c$; we have $e = h \comp c$:
	\[
	\begin{tikzcd}
		A \ar[r, bend left, "f_1"] \ar[r, bend right, swap, "f_0"] & B \ar[r, "c"] \ar[rd, swap, "e"] & C \ar[d, bend left, "h"] \\
		& & C_0 \ar[u, bend left, "m"]
		\end{tikzcd}
	\]
	Then $h = m^{-1}$.
	Indeed, from the above diagram we get
	\[
	(m \comp h) \comp c = m \comp e = c,
	\]
	thus the universality of $c$ yields $m \comp h = \id_C$.
	Since $m$ is monic, this yields $h = m^{-1}$.
	
	\item The proof for directed colimits is completely analogous.

\end{enumerate}
\end{proof}

\begin{lemma}
\label{L:dense}
If a cpo $C$ has an $\omega$-dense subset $X$, then $X^n$ is $\omega$-dense in $C^n$ for every $n \in \Nat$.
\end{lemma}

\begin{proof}
Define a transfinite sequence $X_i$ ($i \in \Ord$) of subsets of $C$ by $X_0 = X$, $X_i = \bigcup_{j < i} X_j$ for limit ordinals, and $X_{i+1}$ being the set of all joins of $\omega$-chains in $X_i$.
The $\omega$-density of $X$ means precisely that $X_\lambda$ is all of $C$ for some ordinal $\lambda$.

Let $Y_i$ be the corresponding sequence in $C^n$, starting with $Y_0 = X^n$.
It is easy to prove by transfinite induction that $Y_i = X_i^n$ for every $i$.
Thus $Y_\lambda$ is all of $C^n$, proving that $Y_0$ is $\omega$-dense. 
\end{proof}

\begin{lemma}
\label{L:cow}
$\CPO$ is cowellpowered with respect to epimorphisms with $\omega$-dense images.
\end{lemma}

\begin{proof}
If $e : A \to B$ has an $\omega$-dense image, let $B_0$ be the smallest subset of $B$ containing $e[A]$ and closed under all existing joins of $B$.
Thus $B_0 = B$ because $e[A] \subseteq B_0$ and $B_0$ is closed under joins of $\omega$-chains.
On the other hand, if $e[A]$ has cardinality $\lambda$, then $B_0$ has cardinality at most $2^\lambda$: take each of the $2^\lambda$ subsets of $e[A]$ and in case it has a join, put that join into $B_0$.
This makes $B_0$ closed under all existing joins of $B$.

If $A$ has cardinality $\kappa$, then $e[A]$ has cardinality $\lambda \leq \kappa$, and $B = B_0$ has cardinality $2^\lambda \leq 2^\kappa$.
There is only a set of cpos of cardinality at most $2^\kappa$ (up to isomorphism).
Thus, $A$ has only a set of quotients represented by epimorphisms with $\omega$-dense images.
\end{proof}

\begin{remark}
\label{R:ADCPO}
All the above has a complete analogy in $\DCPO$:
\begin{enumerate}
	
	\item $\DCPO$ has a factorization system (dense, sub-dcpo).
	
	\item Every coinserter in $\DCPO$ has a dense image, and every directed colimit has a dense union of images.
	
	\item If $X$ is dense in a dcpo $C$, then $X^n$ is dense in $C^n$ for all $n \in \Nat$.
	
	\item $\DCPO$ is cowellpowered with respect to epimorphisms with dense images.
	
\end{enumerate}
All the proofs are completely analogous to the preceding ones.
\end{remark}


\begin{thebibliography}{MM}

\bibitem{adv:ordered-algebras}
J.~Adámek, M.~Dostál and J.~Velebil,
A categorical view of varieties of ordered algebras,
\emph{Math.\ Struct.\ Comput.\ Sci.} (2022), 1--25

%\bibitem{adv:quantitative-algebras}
%J.~Adámek, M.~Dostál and J.~Velebil,
%Quantitative algebras and a classification of metric monads,
%arXiv:2210.01565.

\bibitem{ahs}
J.~Adamek, H.~Herrlich and G.~Strecker,
\emph{Abstract and concrete categories: The joy of cats},
John Wiley and Sons, New York 1990

\bibitem{anr}
J.~Adámek, E.~Nelson and J.~Reiterman,
The Birkhoff variety theorem for continuous algebras,
\emph{Algebra Universalis} 20 (1985), 328-350

\bibitem{adamek+rosicky}
J.~Ad\'{a}mek and J.~Rosick\'{y},
{\em Locally presentable and accessible 
	categories},
Cambridge University Press, 1994

\bibitem{adamek+rosicky:sifted}
J.~Ad\'{a}mek and J.~Rosick\'{y},
On sifted colimits and generalized varieties, 
{\em Theory Appl. Categ.\/} 8 (2001), 33--53

\bibitem{AR23}
J.~Adámek and J.~Rosický,
Varieties of ordered algebras as categories,
\emph{Algebra Universalis} 84 (2023), 1--28

\bibitem{adamek+rosicky+vitale}
J.~Ad\'{a}mek, J.~Rosick\'{y} and E.~Vitale,
{\em Algebraic theories\/},
Cambridge Tracts in Mathematics 184, 2011

\bibitem{ARV10}
J.~Ad\'{a}mek, J.~Rosick\'{y} and E.~Vitale,
What are sifted colimits?,
{\em Theory Appl. Categ.\/} 23 (2010), 251--260

\bibitem{albert+kelly}
M.~H.~Albert and G.~M.~Kelly,
The closure of a class of colimits,
{\em J. Pure Appl. Algebra\/} 51 (1988), 1--17

\bibitem{barr+wells:toposes}
M.~Barr and Ch.~Wells,
\emph{Toposes, triples and theories},
Springer-Verlag, New York 1985

%\bibitem{birkhoff}
%G.~Birkhoff,
%On the structure of abstract algebras,
%{\em Proc.~Camb.~Philos.~Soc.} 31 (1935), 433--454

\bibitem{borceux:second}
F.~Borceux,
\emph{Handbook of Categorical Algebra: Volume 2, Categories and Structures},
Cambridge Univ.\ Press, 1994

\bibitem{BK}
F.~Borceux and G.~M.~Kelly,
A notion of limit for enriched categories,
\emph{Bulletin of the Australian Mathematical Society}, 12(1) (1975), 49--72.

\bibitem{bourke:thesis}
J.~Bourke,
{\em Codescent objects in 2-dimensional 
	universal algebra\/},
PhD Thesis, University of Sydney 2010

\bibitem{bourke-garner:monads-and-theories}
J.~Bourke and R.~Garner,
Monads and theories,
\emph{Adv.~Math.}
351 (2019), 1024--1071

\bibitem{chen}
R.~Chen,
On sifted colimits in the presence of pullbacs,
arXiv:2109.12708

\bibitem{DP}
B.~A.~Davey and H.~Priestley,
Introduction to lattices and order,
Cambridge University Press 2002.

\bibitem{DV}
M.~Dost\'{a}l and J.~Velebil,
An elementary characterisation of sifted 
weights,
arXiv:1405.3090 (2014)

\bibitem{gu}
P.~Gabriel and F.~Ulmer,
{\em Lokal pr\"{a}sentierbare Kategorien\/}, Lecture Notes in Mathematics 221, Springer 1971.

\bibitem{ADJ}
J.~A.~Goguen, E.~G.~Wagner, J.~B.~Wright and J.~W.~Thatcher,
Some fundamentals of order-algebraic semantics,
Proc.~MFCS '76, Lect.~Notes Comput. Sci.~45 (1976),
153--168

\bibitem{PJ}
P.~Johnstone,
\emph{Some aspects of internal category teory in an elementary topos},
PhD Thesis, Cambridge University 1974

\bibitem{Jo}
A.~Joyal,
Notes on quasi-categories,
available at \href{math.uchicago.edu/~may/IMA/Joyal.pdf}{math.uchicago.edu/~may/IMA/Joyal.pdf} (2008).

\bibitem{kelly:book}
G.~M.~Kelly,
{\em Basic concepts of enriched category theory},
London Math.\ Soc.\ Lecture Notes Series 64,
Cambridge Univ.\ Press, 1982,
also available as 
{\em Repr. Theory Appl. Categ.\/} 10 (2005)  

\bibitem{kelly+lack:strongly-finitary}
G.~M.~Kelly and S.~Lack,
Finite-product-preserving functors, Kan 
extensions and
strongly-finitary 2-monads,
{\em Appl. Categ. Structures\/} 1 (1993), 
85--94

\bibitem{kelly+schmitt}
G.~M.~Kelly and V.~Schmitt,
Notes on enriched categories with colimits of some class,
{\em Theory Appl. Categ.\/} 14.17 (2005), 399--423

\bibitem{kurz-velebil:exactness}
A.~Kurz and J.~Velebil,
Quasivarieties and varieties of ordered algebras: regularity and exactness, \emph{Math.\ Structures Comput.\ Sci.} (2016), 1--42.

\bibitem{La}
C.~Lair,
Sur le genre d'esquissabilit\'{e} des cat\'{e}gories 
modelables (accessibles) poss\'{e}dant les produits de deux, 
{\em Diagrammes\/} 35 (1996), 25--52.

\bibitem{Li}
F.~E.~J.~Linton,
Some aspects of equational categories,
\emph{Proc.\ Conf.\ Categorical Algebra in La Jolla}, Springer (1965) 84--94.

\bibitem{maclane:cwm}
S.~Mac Lane,
{\em Categories for the working mathematician\/},
2nd ed., Springer 1998

%\bibitem{mpp17}
%R.~Mardare, P.~Panangaden and G.~D.~Plotkin,
%On the axiomatizability of quantitative algebras,
%\emph{Proceedings of Logic in Computer Science (LICS 2017)}, IEEE Computer Science 2017, 1--12

\bibitem{N}
E.~Nelson,
Z-continuous algebras,
\emph{Lecture Notes in Math.} 871 (1981), 315-334

%\bibitem{JR23}
%J.~Rosický,
%Discrete Lawvere theories and monads,
%prepared for publication.

\end{thebibliography}
\end{document}